\documentclass[final]{siamart1116}

\usepackage{lipsum}
\usepackage{amsfonts}
\usepackage{graphicx}
\usepackage{epstopdf}
\usepackage{algorithmic}
\ifpdf
  \DeclareGraphicsExtensions{.eps,.pdf,.png,.jpg}
\else
  \DeclareGraphicsExtensions{.eps}
\fi

\newcommand{\TheTitle}{More Virtuous Smoothing}
\newcommand{\TheAuthors}{L. Xu, J. Lee, D. Skipper}

\headers{\TheTitle}{\TheAuthors}

\title{{\TheTitle}}

\author{
  Luze Xu\thanks{Department of Industrial and Operations Engineering, University of Michigan, Ann Arbor, MI (\email{xuluze@umich.edu}, \email{jonxlee@umich.edu}).}
  \and
  Jon Lee\footnotemark[1]
  \and
  Daphne Skipper\thanks{Department of Mathematics, U.S. Naval Academy, Annapolis, MD (\email{skipper@usna.edu}).}
}

\usepackage{amsopn}

\ifpdf
\hypersetup{
  pdftitle={\TheTitle},
  pdfauthor={\TheAuthors}
}
\fi

\usepackage{enumitem}
\usepackage{amsfonts}
\usepackage{graphicx}
\usepackage{epstopdf}
\usepackage[caption=false]{subfig}
\usepackage{pdfpages}
\usepackage{bigstrut}
\usepackage{tensor}
\usepackage{amsmath}
\usepackage{hyperref}

\newsiamthm{example}{Example} 
\newsiamremark{remark}{Remark} 
\newsiamremark{observation}{Observation} 
\newsiamthm{jdef}{Definition} 
\crefname{remark}{Remark}{Remarks}
\newcommand{\norm}[1]{\left\Vert#1\right\Vert}
\newcommand{\abs}[1]{\left\vert#1\right\vert}
\newcommand{\hatl}{\hat{\lambda}}

\makeatletter
\newcommand{\labitem}[2]{%
\def\@itemlabel{\textbf{#1}}
\item
\def\@currentlabel{#1}\label{#2}}
\makeatother

\begin{document}

\maketitle

\begin{abstract}
In the context of global optimization of mixed-integer nonlinear optimization
formulations, we consider smoothing univariate functions $f$ that satisfy $f(0)=0$,
$f$ is increasing and concave on $[0,+\infty)$, $f$ is twice differentiable on
all of $(0,+\infty)$, but $f'(0)$ is undefined or intolerably large.
The canonical examples are  root functions $f(w):=w^p$, for $0<p<1$.
We consider the earlier approach of defining a smoothing function $g$ that
is identical with $f$ on $(\delta,+\infty)$, for some chosen $\delta>0$,  then replacing the part of $f$ on $[0,\delta]$
with the unique homogeneous cubic, matching $f$, $f'$ and $f''$ at
$\delta$. The parameter $\delta$  is used to control (i.e., upper bound) the derivative at 0 (which controls it on all of $[0,+\infty)$ when $g$ is concave).
Our main results: (i) we weaken an earlier sufficient condition to give a necessary and sufficient condition for
the piecewise function $g$ to be increasing and concave;
(ii) we give a  general sufficient condition for $g'(0)$ to be decreasing in the smoothing parameter $\delta$; under the same condition, we
demonstrate that the worst-case error of $g$ as an estimate of $f$ is
increasing in $\delta$;
(iii) we give a general sufficient condition for $g$ to underestimate $f$;
(iv) we give a general sufficient  condition for $g$ to dominate the simple `shift smoothing' $h(w):=f(w+\lambda)-f(\lambda)$ ($\lambda>0$), when the parameters $\delta$ and $\lambda$ are chosen ``fairly'' --- i.e., so that $g'(0)=h'(0)$.
In doing so, we solve two natural open problems of Lee and Skipper (2016),
concerning (iii) and (iv) for root functions.
\end{abstract}



\section{Introduction}

\subsection{Motivation}
Most Mixed-Integer Nonlinear Optimization (MINLO) software, aiming at
global optimization of
 so-called factorable math\-ematical\--optimization formulations,
apply the spatial branch-and-bound algorithm  or some close relative of it (e.g.,
\verb;BARON; \cite{sahinidis},
\verb;ANTIGONE; \cite{misener-floudas:ANTIGONE:2014},   open-source
\verb;Couenne; \cite{Belotti09} and   free-for-academic-use
\verb;SCIP; \cite{Achterberg2009}).
As a first step,  problem functions are ``factored'' (i.e., fully decomposed)
via a small library of low-dimensional nonlinear functions (typically, functions
in one, two or three variables) together with affine functions of an arbitrary number of variables.
It is helpful, for robustness, if the library functions are sufficiently smooth over their domains, i.e., typically twice continuously differentiable, so that typical
nonlinear-optimization algorithms may be reliably applied (e.g., \cite{ipopt}).  For functions that are not already sufficiently smooth, it is standard practice for modelers  to replace ``bad'' functions by smoother approximating functions (e.g., \cite{BDLLT06}, \cite{BDLLT12} and \cite{GMS13}). But the issue can also be grappled with algorithmically by (purely continuous) nonlinear-optimization solvers through parameter setting. For example, A. W\"achter explains (see \cite{ipopttutorial}):

\smallskip
\begin{quote}
\emph{``Problem modification:} {\sc Ipopt} seems to perform better if the feasible set of
the problem has a nonempty relative interior. Therefore, by default, {\sc Ipopt}
relaxes all bounds (including bounds on inequality constraints) by a very
small amount (on the order of $10^{-8}$) before the optimization is started. In
some cases, this can lead to problems, and this features can be disabled by
setting \verb;bound_relax_factor; to 0.''
\end{quote}
\smallskip
Consider $f(w):=\sqrt{w}$ on the domain $[0,+\infty)$.
Notice how in this case {\sc Ipopt}'s default value for this parameter \verb;bound_relax_factor; even function cannot be carried out everywhere on the modified domain $[-10^{-8},+\infty)$. And for the suggested nondefault parameter setting (0),
 $\sqrt{w}$ is not differentiable at 0 (in the actual domain). Techniques like smoothly extending $f$ so that $f(w):=-\sqrt{-w}$ for $w<0$ suffer from still not being differentiable at 0.
 So, we are led back to modeling advice (see \cite{ipopttutorial}):

\smallskip
\begin{quote}
``Therefore, it can be
useful to replace the argument of a function with a limited range of definition
by a variable with appropriate bounds. For example, instead of ``$\log(h(x))$'', use
``$\log(y)$'' with a new variable $y\geq \epsilon$ (with a small constant $\epsilon > 0$) and a new
constraint $h(x) - y = 0$.''
\end{quote}
\smallskip
We note that this kind of advice might be problematic in the context of
integer variables, where \emph{precise zero} may be important in constraints implementing
some logic (e.g., see \cite{DFLV2015},\cite{DFLV2014}), and for this reason,
our study is particularly relevant to MINLO.

Notably, the MINLO software \verb;SCIP; has incorporated features
(see \cite{SCIP32}) to accommodate a more sophisticated approach (see \cite{DFLV2015},\cite{DFLV2014},\cite{lee2016virtuous}),  tackling issues of nonsmoothness while at the same time \emph{working within a paradigm that aims at seeking global optimality for nonconvex problems} (see \cite[\S1]{lee2016virtuous}, for details).
Extending the approach of \emph{virtuous smoothing} from \cite[\S1]{lee2016virtuous} is the subject of what follows.


The practical convergence behavior of
the different nonlinear-opt\-i\-mization algorithms that are relied on
by MINLO solvers is the subject of intense investigation;
 see \cite{MKV2017} for a recent experimental comparison.
As we have indicated, the issue of nonsmoothness pertains to an aspect of the
theoretical and practical behavior of various nonlinear-opt\-i\-mization
solvers employed by MINLO solvers.
Our work is aimed at developing
a mathematical framework for improving the behavior.
But, because our interest is especially in global optimization, we seek some control on how solutions employing smooth approximators
relate to solutions employing the functions that they replace --- for example lower or
upper bounding. In cases where our
approximators lower bound the functions that they replace, we can
compare a pair of lower bounding functions when one dominates the other on its domain.


\subsection{Prior work}

The motivating application for our work is root functions $f(w):= w^p$, with $0 < p < 1$, which are smooth everywhere on their domains $[0,+\infty)$, except at $w = 0$.  The inception of this approach is  from \cite{DFLV2015,DFLV2014}, which grappled with handling square-root functions ($p=1/2$) arising in formulations of the Euclidean Steiner Problem.  That successful approach was to replace the part of the root function on $[0,\delta]$, for some small (but not extremely small) $\delta>0$,
with a homogeneous cubic, matching the function and its first two derivatives at
$\delta$. By construction, the new piecewise function $g$ is twice differentiable on
$(0,+\infty)$.
 The parameter $\delta$ is used to control (i.e., upper bound) the derivative at 0.
\cite{DFLV2015,DFLV2014} showed that the new
piecewise function $g$ is (i) increasing and concave, (ii) underestimates the square root, and (iii) dominates the simple shift smoothing $h(w):=\sqrt{w+\lambda}-\sqrt{\lambda}$, when the parameters $\delta$ (for $g$) and $\lambda$ (for $h$) are chosen ``fairly'' --- i.e., so that
$g'(0)=h'(0)$, and hence both smoothing have the same numerical stability.

 In \cite{lee2016virtuous}, we extended this idea of
 \cite{DFLV2015,DFLV2014}, with the following main results:
 \begin{itemize}
 \item[(i)]  a rather general sufficient condition on $f$ (which includes all root functions and more)   so that our smoothing $g$ is increasing and concave;
 \item[(ii)] for root functions of the form $f(w) = w^{1/q}$, with integer $q \geq 2$, our smoothing $g$ underestimates $f$;
 \item[(iii)]  for root functions of the form $f(w) = w^{1/q}$, with integer $2 \leq  q \leq 10,000$, our smoothing $g$ `fairly dominates' the shift smoothing $h$; i.e., when $g$ and $h$ are chosen so that $g'(0)=h'(0)$.
 \end{itemize}
 Regarding (i), the property is useful because we want $g$ to behave like the function $f$ that it replaces. Furthermore, the concavity of $g$ means
 that controlling its derivative at 0 implies that it is controlled on
 all of $[0,+\infty)$. We are now able to extend (i) to get a necessary and sufficient condition.
  Regarding (ii-iii),
the results requiring that $p$ have the form $1/q$ for an integer $q \geq 2$
(and for (iii) even $q\leq 10,000$, which required some computer algebra  for each $q$)
were limited by the algebraic proof techniques  that we employed --- making a transformation to then be able to apply methods that work for analyzing polynomials (e.g., Descartes' Rule of Signs).
We left in \cite{lee2016virtuous} as substantial open problems extending (ii-iii) to \emph{all} root functions.
In what follows, we resolve these open problems and generalize the theorems
quite a bit further, by employing methods of analysis instead of algebraic methods.

\subsection{Definition of \texorpdfstring{$\delta$-smoothing}{delta-smoothing}}

 Let $f$ be a univariate function having a domain  $I:=[0,U)$, where $U \in \{ w\in\mathbb{R} ~:~ w > 0\} \cup \{+\infty\}$. Suppose that $\delta>0$ is in the domain of $f$.

\begin{jdef}
{\rm We say that such an $f$ satisfies the} minimal $\delta$-smoothing requirements
{\rm if $f(0)=0$, and $f$ is twice differentiable at $\delta$.}
\end{jdef}

 In the spirit of \cite{lee2016virtuous} (though we note that they always assumed $U=+\infty$), we will define a ``$\delta$-smoothing'' of $f$.

\begin{jdef}
{\rm Suppose that such an
 $f$ satisfies the minimal $\delta$-smoothing requirements.
Then the} $\delta$-smoothing {\rm of $f$ is the piecewise-defined function
\[
g(w) := \left\{
\begin{array}{ll}
  g_1 w + \frac{1}{2} g_2 w^2 + \frac{1}{6} g_3 w^3, &  0 \leq w \leq \delta; \\
  f(w), & \delta < w < U,
\end{array}
\right.
\]
with
\begin{displaymath}
\begin{array}{rl}
&g_1:=\displaystyle\frac{3f(\delta)}{\delta}-2f'(\delta)+\frac{\delta f''(\delta)}{2};\\[5pt]
&g_2:=\displaystyle-\frac{6f(\delta)}{\delta^2}+\frac{6f'(\delta)}{\delta}-2f''(\delta);\\[5pt]
&g_3:=\displaystyle \frac{6f(\delta)}{\delta^3}-\frac{6f'(\delta)}{\delta^2}+\frac{3f''(\delta)}{\delta}.
\end{array}
\end{displaymath}
}
\end{jdef}
\bigskip

Obviously the function $g$ and its coefficients $g_1,g_2,g_3$
depend  on $\delta$,
but to keep the notation uncluttered, we do not indicate this in the notation.

Although the coefficients $g_i$ (in the cubic portion of $g$) have a rather complicated specification, it is easy to check that the cubic portion of $g$
 is the unique minimum-degree polynomial having:
\begin{align*}
g(0) &= f(0)~=~0; \\
g(\delta) &= f(\delta); \\
g'(\delta) &= f'(\delta); \\
g''(\delta) &= f''(\delta).
\end{align*}

We chose the precise form of the homogeneous cubic, for later convenience, so that the coefficients $g_i$ satisfy:
\begin{eqnarray*}
g_1 &=& g'(0); \\
g_2 &=& g''(0); \\
g_3 &=& g'''(w), \mbox{ for } w \in [0,\delta].
\end{eqnarray*}

As in \cite{lee2016virtuous}, our main motivation is situations in which, like root functions, $f'(0)$ is undefined or intolerably large. As we will see, the parameter $\delta$ is used to control the derivative of $g$ at $0$.  Also motivated by root functions, we are particularly interested in functions $f$ that are continuous, increasing, and concave on their domains.  
MINLO solvers like \verb;BARON;, \verb;SCIP; and \verb;ANTIGONE; are improving their performance
by growing their set of low-dimensional library functions,
as a means of getting stronger relaxations. This can lead to stronger relaxations
than simply combining relaxations across function compositions.
So we seek general
methods for smoothing that can be readily applied. Although our first
challenging motivation is root functions, there are other natural functions
that occur for which our methods apply. For example, the concave \emph{entropy function}
\[
f(w) := \left\{
          \begin{array}{ll}
            -w \log(w), & \hbox{$0<w\leq 1$;} \\
            0, & \hbox{w=0}
          \end{array}
        \right.
\]
is continuous on $[0,1]$, but its derivative blows up at 0.
Our methods apply here, and we have plans to implement our smoothing
for it
in \verb;ANTIGONE;. Another example is the concave and increasing \emph{incremental entropy
function}
\[
f(w) := \left\{
          \begin{array}{ll}
            w \log(1+\frac{1}{w}), & \hbox{$w>0$;} \\
            0, & \hbox{w=0.}
          \end{array}
        \right.
\]

To eliminate any chance of confusion, throughout, for a real interval $I$ and a  function $\phi: I \to \mathbb{R}$, $f$ is \emph{increasing} if $\phi(w_1)<\phi(w_2)$ for all $w_1<w_2\in I$,
and \emph{nondecreasing} if $\phi(w_1)\leq \phi(w_2)$ for all $w_1<w_2\in I$.
Similarly for \emph{decreasing} and \emph{nonincreasing}.
The function $\phi$ is \emph{strictly concave} if
$\phi(\lambda w_1 +(1-\lambda)w_2) >
\lambda \phi(w_1) +(1-\lambda)\phi(w_2)$, for all $w_1<w_2\in I$,
and all $0<\lambda <1$, and
\emph{concave} if
$\phi(\lambda w_1 +(1-\lambda)w_2) \geq
\lambda \phi(w_1) +(1-\lambda)\phi(w_2)$, for all $w_1<w_2\in I$,
and all $0<\lambda <1$. Similarly for \emph{strictly convex} and \emph{convex}.

In \S\ref{sec:tendency}, we weaken the sufficient condition \cite[Theorem 2]{lee2016virtuous} to now give a necessary and sufficient condition for $g$ to be increasing and concave.   We also provide conditions under which $g'(0)$ has desirable behaviors.  In \S\ref{sec:lowb}, we give a sufficient condition for $g$ to underestimate $f$, greatly generalizing \cite[Theorem 9]{lee2016virtuous}.
Additionally, in \S\ref{sec:lowb},  we
analyze the dependence on our smoothing parameter $\delta$, of the worst-case  behavior of $g$ as an approximation of $f$.
  In \S\ref{sec:beb}, we give a general sufficient  condition for $g$ to dominate the simple `shift smoothing' $h(w):=f(w+\lambda)-f(\lambda)$ ($\lambda>0$), when the parameters $\delta$ (for $g$)  and $\lambda$
  (for $h$) are chosen ``fairly'' --- i.e., so that $g'(0)=h'(0)$, greatly generalizing \cite[Theorem 10]{lee2016virtuous}.
 Via our  main results in \S\ref{sec:lowb} and  \S\ref{sec:beb},
we solve two natural open problems of  \cite{lee2016virtuous}
concerning root functions, and in fact extend those results
significantly beyond root functions.
  In \S\ref{sec:conclusions}, we make some brief concluding remarks.


\section{General behaviors of \texorpdfstring{$\delta$-smoothing}{delta-smoothing}}
\label{sec:tendency}
In this section, we explore general properties of $\delta$-smoothings that are not directly related to bounding $f$ (which we will take up in \S\ref{sec:lowb}).  In \S\ref{sec:ic} we provide a necessary and sufficient condition on an increasing and concave $f$ under which its $\delta$-smoothing $g$ is also increasing and concave.  In \S\ref{sec:property}, we provide properties relating the behaviors of $g'$ and $f'$ near zero, when $f'''$ is decreasing.  In \S\ref{monotone}, we show that $f'''$ being decreasing  is a sufficient condition for $g_1 = g'(0)$ to be decreasing in the smoothing parameter $\delta$ --- a property which is practically useful in choosing a good value for $\delta$.

\subsection{Increasing and concave}\label{sec:ic}
In the context of
global optimization, it is desirable for the $\delta$-smoothing $g$ of a function $f$
to share properties with $f$ beyond those inherent in the definition of $g$.
For example, when $f$ is a root function, $f$ is increasing and  concave.
In this way, $g$ can be algorithmically treated by global-optimization software in a way that is consistent
with the treatment of $f$ (e.g., tangents for overestimating and secants for
underestimating). Furthermore, concavity of $g$ implies that controlling
\[
g'(0) = g_1 = \displaystyle\frac{3f(\delta)}{\delta}-2f'(\delta)+\frac{\delta f''(\delta)}{2}
\]  (by choosing $\delta>0$ appropriately) has the effect of
controlling $g'(w)$ on all of its nonnegative domain.

In \cite{lee2016virtuous}, we gave the following lower bound
 on the negative curvature of $f$ at $\delta$ as a sufficient condition for $g$ to be increasing and concave on $[0,\delta]$.

\begin{theorem}\label{thm:ic}(\cite[Theorem 2]{lee2016virtuous})
Let $f$ be a univariate function having a domain  $I:=[0,U)$, where $U \in \{ w\in\mathbb{R} ~:~ w > 0\} \cup \{+\infty\}$. Suppose that $\delta>0$ is in the domain of $f$. Assume that $f$ satisfies the minimal $\delta$-smoothing requirements.
 Suppose further that
\begin{itemize}
\item $f$ is increasing and differentiable on $[\delta, U)$;
\item $f'$ is nonincreasing (resp., decreasing) on $[\delta, U)$.
\end{itemize}
If
\begin{equation}  \label{eq:Tdelta}
f''(\delta) \geq \frac{2}{\delta}\left( f'(\delta)-\frac{f(\delta)}{\delta}\right)
~~\big( \Leftrightarrow g_3 \geq 0 ~\big),
\tag{\text{$T_\delta$}}
\end{equation}
then the $\delta$-smoothing $g$ of $f$  is increasing and concave (strictly concave) on  $[0,U)$.
\end{theorem}

If we make the further mild assumption that $f$ is differentiable on $(0,\delta)$, then by Rolle's Theorem,
there is a $u\in(0,\delta)$ so that
\[
\frac{f(\delta)-f(0)}{\delta-0} = f'(u).
\]
If we make the still further mild assumption that $f'$ in nonincreasing on $(0,\delta]$, then we can conclude that
$f'(\delta)-\frac{f(\delta)}{\delta}\leq 0$.
Then the intuition for \eqref{eq:Tdelta} is that if there is not too much negative curvature of $f$ at $\delta$,
 then the function can make it to the origin staying increasing and concave.



This sufficient condition is met by all root functions and more (see \cite[Examples 6,7]{lee2016virtuous}).
Of course we may be concerned that the sufficient condition \eqref{eq:Tdelta} is too strong, and
the following example demonstrates that \eqref{eq:Tdelta} is not necessary for $g$ to be increasing and concave.
\noindent
\begin{example}
\label{ex:AB}
For $\epsilon>0$, let
\begin{displaymath}
f(w) := \left \{
\begin{array}{ll}
\displaystyle-\frac{1}{24\epsilon\sqrt{\epsilon}(1+\epsilon)}w^3+\frac{1+5\epsilon}{8\epsilon\sqrt{\epsilon}}w, &0\leq w\le 1+\epsilon;\\
\sqrt{w-1} + \displaystyle \frac{1+8\epsilon-5\epsilon^2}{12\epsilon\sqrt{\epsilon}},& w > 1+\epsilon.
\end{array}
\right .
\end{displaymath}
It is straightforward to verify that $f(0)=0$,  and that $f$ is twice-differentiable, increasing, and concave on $[0,+\infty)$.

Now, let $\delta:=1$. Because $f$ is a cubic function on $[0,\delta]$, $g$ is the same as $f$ on $[0,\delta]$, which means that $g$ is also increasing and concave. However, in this case, $g'''(\delta)=\displaystyle -\frac{1}{4\epsilon\sqrt{\epsilon}(1+\epsilon)}<0$, contradicting \eqref{eq:Tdelta}.

Finally, one could argue that this example is unfair, because we are not actually smoothing anything at 0. But, following the idea in \cite[Section 2.2]{lee2016virtuous}, we could add a very small positive multiple of $\sqrt{w}$ to this $f(w)$,
and then we would get a legitimate example, nonsmooth at 0.\hfill  $\Diamond$
\end{example}

Next, we will see that by ``weakening the condition  \eqref{eq:Tdelta} by 50\%'' (see the paragraph after  \eqref{eq:Tdelta}),
we obtain a necessary and sufficient condition for $g$ to be increasing and concave on $[0,\delta]$.  In fact, we will see that the condition is precisely motivated by \cref{ex:AB}.

\begin{theorem}\label{thm:icexact}
Let $f$ be a univariate function having a domain  $I:=[0,U)$, where $U \in \{ w\in\mathbb{R} ~:~ w > 0\} \cup \{+\infty\}$. Suppose that $\delta>0$ is in the domain of $f$. Assume that $f$ satisfies the minimal $\delta$-smoothing requirements.
 Suppose further that
\begin{itemize}
\item $f$ is increasing and differentiable on $[\delta, U)$;
\item $f'$ is nonincreasing (resp., decreasing) on $[\delta, U)$.
\end{itemize}
Then $g$ is increasing and concave (strictly concave)
on  $[0,U)$ if and only if
\begin{equation}  \label{eq:TSdelta}
f''(\delta) \geq \frac{3}{\delta}\left( f'(\delta)-\frac{f(\delta)}{\delta}\right)
~~\big( \Leftrightarrow g_2 \leq 0 ~\big).
\tag{\text{$T^*_\delta$}}
\end{equation}
\end{theorem}

\begin{proof}
Necessity is obvious because $g''(0)=g_2\leq 0$. For  sufficiency, under \eqref{eq:TSdelta}, we have $g''(0)\leq 0$. Along with $g''(\delta)=f''(\delta)$ is nonpositive (negative) and that $g''(w)$ is linear (in $w$),
we have $g'(w)$ is nonincreasing (decreasing) to $g'(\delta)=f'(\delta)>0$, therefore $g$ is concave (strictly concave) and increasing on $[0,\delta]$. Note that the assumptions on $f$ imply that $g$ is concave (strictly concave) and increasing on $[\delta,U)$, the conclusion follows.
\end{proof}

As the function in \cref{ex:AB}  has $g_2=0$,
it satisfies property \eqref{eq:TSdelta} as an equation.

\subsection{Controlled derivative at 0}\label{sec:property}
The primary goal of $\delta$-smoothing is to approximate $f$ by  a smooth function $g$ having derivative controlled at zero.  In \cref{prop:thm1}, we present several properties relating the behaviors of derivatives of $f$ and $g$ at the ends of the interval $[0,\delta]$ in the event that $f'''$ exists and is decreasing on   $(0,\delta]$.  As we will see, for increasing and concave $f$, we get conditions under which both the first and second derivatives of $g$ are more controlled near zero than those of $f$.  Looking ahead, \cref{prop:thm1} will also be used in \S\ref{monotone} and \S\ref{fgmonotone} to prove the monotonicity of $g_1$ and $\norm{f-g}_{\infty}$ in the smoothing parameter $\delta$, and in \S\ref{sec:beb} to demonstrate that $g$ is a tighter lower bound for ``root-like functions'' than the natural ``shift smoothing''.
\begin{proposition}
\label{prop:thm1}
Let $f$ be a univariate function having a domain  $I:=[0,U)$, where $U \in \{ w\in\mathbb{R} ~:~ w > 0\} \cup \{+\infty\}$. Suppose that $\delta>0$ is in the domain of $f$. Assume that $f$ satisfies the minimal $\delta$-smoothing requirements.
 Suppose further that
\begin{itemize}
	\item $f$ is continuous on $[0,\delta]$ and thrice differentiable on $(0,\delta]$,
	\item $f'''$ is decreasing on $(0,\delta]$.
\end{itemize}
Then $f$ has the following properties:
\smallskip
\begin{enumerate}[label=(\arabic*),leftmargin=5\parindent,itemsep=1ex]
\item \label{p1} $\displaystyle\lim_{w\to 0^+}{f'(w)}> g_1 = g'(0)$;
\item \label{p2} $\displaystyle\lim_{w\to 0^+}{f''(w)}< g_2 = g''(0)$;
\item \label{p3} $\displaystyle\lim_{w\to 0^+}{f'''(w)}> g_3 = g'''(0)$;
\item \label{p4} $f'''(\delta)<g_3$.
\end{enumerate}
\end{proposition}

\begin{proof}
Clearly $f\not=g$ because $f'''$ is decreasing on $[0,\delta]$ while $g'''$ is constant on $[0,\delta]$.
Define $F:=f-g$ on $[0,\delta]$, and let $J := (0,\delta)$.  Then $F(0)=0$ and $F^{(i)}(\delta)=0$ for $i=0,1,2$.    Because $f'''$ is decreasing on $(0,\delta]$, $F''' = f'''-g_3$ is also decreasing on $(0,\delta]$.

Suppose property \ref{p4} does not hold, i.e., $F'''(\delta)\ge0$. Then on $J$, $F''' > 0$ or equivalently, $F''$ is increasing. Because $F''(\delta)=0$, $F'' < F''(\delta)=0$ and $F'$ is decreasing on $J$. Because $F'(\delta)=0$, $F' > F'(\delta)=0$ and $F$ is increasing on $J$. Noting that $F(0)=F(\delta)=0$, we have $F \equiv 0$; i.e., $f=g$.

Suppose property \ref{p3} does not hold, i.e., $\lim\limits_{w\rightarrow0^+}F'''(w)\le0$, so that $F''' \le \lim\limits_{w\rightarrow0^+}F'''(w)\le0$ on $J$. Following a similar argument as above, on interval $J$, $F''$ is decreasing, $F'$ is increasing, and $F$ is decreasing.  Again we arrive at the trivial case: $f = g$.

Suppose property \ref{p2} does not hold, i.e., $\lim\limits_{w\rightarrow0^+}F''(w)\ge0$.  From properties \ref{p3} and \ref{p4}, we know that $F''$ is first increasing and then decreasing on $J$. Thus, $F'' \ge 0$ on $J$. As above, we find that $F'$ is increasing and $F$ is decreasing on $J$, leading again to the trivial case.

Suppose property \ref{p1} does not hold, i.e., $\lim\limits_{w\rightarrow0^+}F'(w)\le0$.
Property \ref{p2}, along with the facts that $F''$ is first increasing and then decreasing on $J$ and $F''(\delta)=0$, implies that $F'$ is first decreasing and then increasing on $J$. Therefore, $F' \le 0$, and $F$ is nonincreasing on $J$, so that $f = g$.
\end{proof}

When $f$ is increasing and concave and $g_2\le0$, $g$ is increasing and concave by \cref{thm:icexact}. In this case, property \ref{p1} implies that $g'$ is more controlled near $0$ than $f'$, and property \ref{p2} implies that $-g''$ is more controlled near $0$ than $-f''$. Of course, via $\delta$ we have control over both $g'(0)$ and $-g''(0)$.

\subsection{Monotonicity of \texorpdfstring{$g_1=g'(0)$ in $\delta$}{g'(0) in delta}}\label{monotone}

For a particular increasing and concave $f$, it may seem intuitive that $g_1=g'(0)$ should be  decreasing in the smoothing parameter $\delta$, for $\delta>0$ in the domain of $f$. This would be a very useful property, because then we could easily find a value for $\delta$ to achieve a target value for
$g_1$ using a simple univariate search.  As we explore the tendency of $g_1$ with respect to $\delta$, it is useful to emphasize the functional dependence of  $g_1$ on $\delta$ by writing  $g_1(\delta)$.

It is straightforward to calculate the derivative of this function:
\begin{equation*}
\frac{d g_1(\delta)}{d\delta}
~=~
- \frac{3}{\delta^2} f(\delta)
+\frac{3}{\delta} f'(\delta)
-\frac{3}{2}f''(\delta)
+ \frac{\delta}{2} f'''(\delta).
\end{equation*}
Unfortunately, for concrete functions $f$, it may not be so practical to check that this derivative is nonpositive  for $\delta>0$ in the domain of $f$.
So, to establish such monotonicity in a \emph{practically verifiable manner}, we need to make some appropriate hypotheses.

\begin{theorem}\label{thm:monotone}
Let $f$ be a univariate function having a domain  $I:=[0,U)$, where $U \in \{ w\in\mathbb{R} ~:~ w > 0\} \cup \{+\infty\}$.  Assume that $f$ satisfies the minimal $\delta$-smoothing requirements for all $\delta>0$ in the domain of $f$.
Suppose further that

\begin{itemize}
  \item $f$ is continuous on $[0,U)$ and thrice differentiable on $(0,U)$;
  \item $f'''$ is decreasing on $(0,U)$.
\end{itemize}
Then $g_1(\delta)$ is decreasing on $(0,U)$.
\end{theorem}


\begin{proof}
It is easy to check that
\begin{equation*}
\frac{d g_1(\delta)}{d\delta} ~=~ \frac{\delta}{2}(f'''(\delta)-g_3(\delta)).
\end{equation*}

We want $f'''(\delta)-g_3(\delta)< 0$ on $(0,U)$, so we can
conclude that $g_1(\delta)$ is decreasing on $(0,U)$.
For a fixed $\delta\in(0,U)$, by \cref{prop:thm1}, we have $f'''(\delta)-g_3(\delta)<0$, which gives us
$g'_1(\delta)< 0$
on $(0,U)$.
\end{proof}

Applying \cref{thm:monotone}, it is now a simple matter to verify that when $f$ is a root function, $g'(0)$ behaves as expected with respect to parameter $\delta$.
\begin{corollary} \label{cor:incCwp}
Let $f(w): = w^p$, for some $0<p<1$.
 Then $g_1(\delta)$ is decreasing on $(0,+\infty)$.
\end{corollary}

\begin{proof}
We must verify that $f$ satisfies the hypothesis of \cref{thm:monotone}. Consider the following derivatives of $f$ on $(0,+\infty)$:
\begin{align*}
f'(w) &= pw^{p-1};\\
f''(w) &= p(p-1)w^{p-2};\\
f'''(w) &= p(p-1)(p-2)w^{p-3};\\
f^{(4)}(w) &= p(p-1)(p-2)(p-3)w^{p-4}.
\end{align*}
Because $0<p<1$, $f^{(4)}(w)<0$ on $(0,+\infty)$, which implies $f'''$ is decreasing on $(0,+\infty)$, thus \cref{thm:monotone} applies.
\end{proof}

The next example demonstrates that \cref{thm:monotone} applies to functions that are not root functions.

\begin{example} \label{ex:incCSH}
Let $f(w) := {\rm ArcSinh}(\sqrt{w})=\log\left(\sqrt{w}+\sqrt{1+w}\right)$, for $w \geq 0$.
Checking the hypotheses of Theorem \ref{thm:monotone}, we calculate the following derivatives of $f$ on $(0,+\infty)$:
\begin{align*}
f'(w) &= \frac{1}{2\sqrt{w(w+1)}};\\
f''(w) &= -\frac{2w+1}{4\left(w(w+1)\right)^{\frac32}};\\
f'''(w) &= \frac{8w^2+8w+3}{8\left(w(w+1)\right)^{\frac52}};\\
f^{(4)}(w)&= -\frac{48w^3+72w^2+54w+15}{16\left(w(w+1)\right)^{\frac72}}.
\end{align*}
For $w\in(0,+\infty)$, it is easy to verify that $f^{(4)}(w)<0$, which implies $f'''$ is decreasing on $(0,+\infty)$. By \cref{thm:monotone},
$g_1(\delta)$ is decreasing for $\delta \in (0,+\infty)$.
\end{example}

\section{Lower bound for \texorpdfstring{$f$}{f}}
\label{sec:lowb}
In \S\ref{lowerbound}, we establish \cref{thm:1}: $g$ provides a lower bound for a broad class of functions $f$ which includes all root functions, solving an open problem from \cite{lee2016virtuous}.
We provide an example to demonstrate that this class of functions contains  functions
beyond root functions. In \S\ref{morepossibilities}, we present variations on the hypotheses of \cref{thm:1}, along with supporting examples.
 In \S\ref{roleofTdelta}, we veer briefly from root-like functions to provide an example of a function $f$ that is neither increasing nor concave, but for which $g$ serves as a lower bound. In other words, we show that \cref{thm:1} does not require $f$ to be increasing and concave. Also we give an example to show that for an increasing and concave function $f$, \eqref{eq:TSdelta} is not necessary for \cref{thm:1}. In \S\ref{fgmonotone}, we demonstrate  that
the worst-case error of $g$ as an approximation of $f$
 is increasing with respect to $\delta$ under the same conditions as \cref{thm:monotone}.

\subsection{Lower bounding}\label{lowerbound}

Because the $\delta$-smoothing $g$ is simply $f$ on $(\delta,U)$, we restrict our attention to lower bounding on the interval $[0,\delta]$.
The parameter $\delta$  provides control over $g'(0)$,
and in a predictable manner under the hypotheses of \cref{thm:monotone}.
As $\delta$ vanishes, $g$ tends to $f$, but the choice of $\delta$ is dictated by the numerical tolerance of the software with respect to the value of $g'(0)$. The following theorem shows that $g$ provides a lower bound for a broad class of functions $f$ which is neither necessarily increasing nor concave (examples are in \S\ref{roleofTdelta}), but includes all root functions.


\begin{theorem}
\label{thm:1}
Let $f$ be a univariate function having a domain  $I:=[0,U)$, where $U \in \{ w\in\mathbb{R} ~:~ w > 0\} \cup \{+\infty\}$. Suppose that $\delta>0$ is in the domain of $f$. Assume that $f$ satisfies the minimal $\delta$-smoothing requirements.
Assume further that
\begin{itemize}
	\item $f$ is continuous on $[0,\delta]$;
	\item $f'''$ exists and is decreasing on $(0,\delta]$.
\end{itemize}
Then $g(w)<f(w)$ for all $w\in (0,\delta)$.
\end{theorem}


\begin{proof}
This is a special case of ``osculating interpolation'' (also know as Hermite interpolation; see \cite{burden2011numerical}, for example). We are going to use the technique of error estimation for osculating interpolation to prove that
\[
K(w)=\frac{f(w) - g(w)}{w(w-\delta)^3}<0,~\text{for}~w\in (0,\delta).
\]

For some fixed $w\in(0,\delta)$, denote $K:=K(w)$ for simplicity, and introduce a new function $F$ with respect to $x$ as
\[
F(x) := f(x) - g(x) -Kx(x-\delta)^3.
\]
By the definition of $K$, we have $F(w)=0$. Also from the relationships between $f$ and $g$, we have $F(0)=F(\delta)=F'(\delta)=F''(\delta)=0$. It is easy to see that $0,w,\delta$ are three zeros for $F(x)$. Because $F(x)$ is continuous on $[0,\delta]$ and differentiable on $(0,\delta)$, according to Rolle's Theorem, there exists $0<w_1<w<\eta_1<\delta$ such that $F'(w_1)=F'(\eta_1)=0$. Noting that $F'(\delta)=0$ and that $F'(x)$ is differentiable on $[w_1,\delta]$, we apply Rolle's Theorem and get $w_1<w_2<\eta_1<\eta_2<\delta$ such that $F''(w_2)=F''(\eta_2)=0$. Using Rolle's Theorem again on $F''(x)$ with $F''(\delta)=0$ and $F''(x)$ is differentiable on $[w_2,\delta]$, we get $w_2<w_3<\eta_2<\eta_3<\delta$ such that $F'''(w_3)=F'''(\eta_3)=0$.

Now, $F'''(x)=f'''(x) - g_3 - K(24x-18\delta)$.  Applying $F'''(w_3)=F'''(\eta_3)$ and $f'''(w_3)> f'''(\eta_3)$, we can conclude that $K(24w_3-18\delta)> K(24\eta_3-18\delta)$.  But this last inequality holds only when $K < 0$.
\end{proof}

\bigskip

It is easy to see that  \Cref{thm:1} has a counterpart when $f'''$ is increasing rather than decreasing, by applying \Cref{thm:1}  to $-f$.

\begin{corollary}
\label{prop:ub}
Let $f$ be a univariate function having a domain  $I:=[0,U)$, where $U \in \{ w\in\mathbb{R} ~:~ w > 0\} \cup \{+\infty\}$. Suppose that $\delta>0$ is in the domain of $f$. Assume that $f$ satisfies the minimal $\delta$-smoothing requirements.
Assume further that
\begin{itemize}
	\item $f$ is continuous on $[0,\delta]$;
	\item $f'''$ exists and  is increasing on $(0,\delta]$.
\end{itemize}
Then  $f(w)< g(w)$ for all $w\in (0,\delta)$.
\end{corollary}

Returning to our primary motivation, the following corollary demonstrates that \cref{thm:1} generalizes the result  in \cite{lee2016virtuous}, which states that $g$ is a lower bound for root functions of the form $f(w) = w^{1/q}$, for integer $q \geq 2$.

\begin{corollary} \label{cor:wp}
Let $f(w): = w^p$, for some $0<p<1$.
For all $\delta > 0$, if $g$ is the $\delta$-smoothing of $f$, then $g(w) \leq f(w)$, for $w \geq 0$.
\end{corollary}

\begin{proof}
According to \cref{cor:incCwp}, we can simply verify that $f'''$ is decreasing on $(0,\delta]$,
thus \cref{thm:1} applies.
\end{proof}

The next example demonstrates that there are other increasing and concave functions (besides root functions) to which \cref{thm:1} applies.

\begin{example} \label{ex:log}
Consider $f(w) :={\rm ArcSinh}(\sqrt{w})=\log\left(\sqrt{w}+\sqrt{1+w}\right)$, for $w \geq 0$.  We demonstrate that $f$ satisfies the conditions of \cref{thm:1}, so that $g$ lower bounds $f$ on $[0,+\infty)$.
From \cref{ex:incCSH}, we can easily verify that $f'''$ is decreasing on $(0,\delta]$, thus $f$ satisfies the conditions of \cref{thm:1}.
\hfill $\Diamond$
\end{example}
\bigskip

\subsection{More possibilities for a lower bound}\label{morepossibilities}
We digress again to provide results that take us beyond root functions.  In particular,
there are other possibilities for $f'''$ (besides decreasing) to ensure that $g$ is a lower-bound on $f$.
For example, in \cref{thm:1_gen} below, if we have $f'''$ first decreasing and then increasing on $(0,\delta]$, we can add conditions almost identical to properties \ref{p1}-\ref{p4} of \cref{prop:thm1}  to ensure a lower-bounding $g$.

\begin{theorem}\label{thm:1_gen}
Let $f$ be a univariate function having a domain  $I:=[0,U)$, where $U \in \{ w\in\mathbb{R} ~:~ w > 0\} \cup \{+\infty\}$. Suppose that $\delta>0$ is in the domain of $f$. Assume that $f$ satisfies the minimal $\delta$-smoothing requirements.
Assume further that
\begin{itemize}
	\item $f$ is continuous on $[0,\delta]$ and thrice differentiable on $(0,\delta]$;
	\item $f'''$ is first decreasing and then increasing on $(0,\delta]$.
\end{itemize}
Moreover, suppose that
\begin{enumerate}[label=(\arabic*),leftmargin=5\parindent,itemsep=1ex]
\item $\displaystyle\lim_{w\to 0^+}{f'(w)}> g_1$;
\item $\displaystyle\lim_{w\to 0^+}{f''(w)}< g_2$;
\item $\displaystyle\lim_{w\to 0^+}{f'''(w)}> g_3$;
\item[($4^{{\leq}}$)] $f'''(\delta)\le g_3$,
\end{enumerate}
\smallskip
then $f(w)\geq g(w)$ for all $w\in [0,+\infty)$.
\end{theorem}

\begin{proof}
According to the definition of $g$, we have $g(0)=0$, $g^{(i)}(\delta)=f^{(i)}(\delta)$, for $i=0,1,2$. We consider the function $F(w) := f(w)- g(w)$, for $w \in [0,\delta]$, which has
\[F(0)=F(\delta)=F'(\delta)=F''(\delta)=0.\]
In what follows, we begin with the third derivative of $F$ and work our way to the conclusion that $F(w) > 0$ for $w \in (0,\delta)$.

First, we note that
$F'''(w) = f'''(w)-g_3$ is a first decreasing and then increasing function with
\[
\lim_{w\to 0^+}F'''(w)
>0 \text{~~~~and~~~~}  F'''(\delta)
\le0.
\]
Therefore, there exists exactly one root of $F'''$ in $(0,\delta)$, which we denote by $w_0$.

From this, we conclude that $F''$ is increasing on $[0,w_0]$ and decreasing on $[w_0,\delta]$, so that $F''(w_0)>F''(\delta)=0$.  Combining this fact with
\[
\lim_{w\to 0^+}F''(w)=\lim_{w\to 0^+}(f''(w)-g_3w-g_2)
=\lim_{w\to 0^+}(f''(w)-g_2)<0,
\]
we see that $F''$ has exactly one root in $(0,w_0)$, which we denote by $w_1$.  In summary, we have
\[
F''(w)~~
\left \{
\begin{array}{ll}
<~0, & 0 < w < w_1; \\
=~0, & w\in \{w_1,\delta\}; \\
>~0, & w_1 < w < \delta. \\
\end{array}
\right.
\]

Applying these results, we conclude that $F'$ is decreasing on $[0,w_1]$ to a minimum of $F'(w_1)<F'(\delta) = 0$. Because
\[\lim_{w\to 0^+}F'(w)=\lim_{w\to 0^+}(f'(w)-\frac{1}{2}g_3w^2-g_2w-g_1)
=\lim_{w\to 0^+}(f'(w)-g_1)>0,\]
we see that $F'$ has exactly one root in $(0,w_1)$, which we denote by $w_2$, and
\[
F'(w)~~
\left \{
\begin{array}{ll}
>~0, & 0 < w < w_2; \\
=~0, & w\in \{w_2,\delta\}; \\
<~0, & w_2 < w < \delta.
\end{array}
\right.
\]

By properties of its derivative, $F(w)$ is increasing on $[0,w_2]$ and decreasing on $[w_2,\delta]$. Because $F(0)=F(\delta)=0$, we have that $F(w) = f(w) - g(w) >0$ for $w\in (0,\delta)$.  Recalling that $f(w) = g(w)$ for $w \in  \{0\} \cup [\delta, \infty)$, we conclude that $g\leq f$ on $[0,+\infty)$.
\end{proof}

\bigskip
\begin{remark}\label{rk:w}

If $f'''$ is decreasing, then the hypotheses of \cref{thm:1} imply properties \ref{p1}-\ref{p4} of \cref{prop:thm1}. By employing these properties, we can use the same proof technique from \cref{thm:1_gen} to prove \cref{thm:1}.
As in the proof of \cref{thm:1_gen}, we can prove $f\ge g$ by considering the function $F:=f-g$.  The third derivative, $F'''(w) = f'''(w)-g_3$, is decreasing with $\lim_{w\to 0^+}F'''(w)=\lim_{w\to 0^+}f'''(w)-g_3>0$ and $F'''(\delta)=f'''(\delta)-g_3<0$.   Therefore, there exists exactly one root of $F'''$ in $(0,\delta)$, which we denote by $w_0$. The rest of the proof is the same as that of \cref{thm:1_gen}. From the proof, we can find the roots $w_0,w_1,w_2$ of the derivatives of the function $F$ and the same characterization for the derivatives as \cref{thm:1_gen}.  We require this characterization in the proof of \cref{thm:fgmonotone} and \cref{thm:fg_gen}.

\end{remark}

In order to demonstrate the applicability of \cref{thm:1_gen}, we construct \cref{ex:thm1_gen} using the general form described in \cref{ex:fg} below.  Inspired by \cref{ex:AB}, we build a continuous piecewise-defined function specified as a quintic on $[0,w_0]$, and a shifted square root function on $(w_0,+\infty)$.  We will use the same general form again in \cref{ex:gg}. 
\bigskip

\begin{example} \label{ex:fg}
Consider the function
\begin{displaymath}
f(w) := \left \{
\begin{array}{ll}
a_5w^5+a_4w^4+a_3w^3+a_2w^2+a_1w, & 0\leq w\leq w_0;\\
a\sqrt{w-c}+b,&w > w_0.
\end{array}
\right .
\end{displaymath}
After fixing the values of the parameters $\delta$, $w_0$, $a_2$, $a_3$, $a_4$, and $a_5$ so that $\frac{f''(w_0)}{f'''(w_0)} \leq 0$,  we ensure continuity and thrice differentiability of $f$ at $w_0$ by calculating the remaining parameters as follows:
\begin{align*}
c &= w_0 + \frac{3f''(w_0)}{2f'''(w_0)};\\
a_1 &= -2f''(w_0)(w_0-c) - (5a_5w_0^4+4a_4w_0^3+3a_3w_0^2+2a_2w_0);\\
a &= \frac{8f'''(w_0)(w_0-c)^{\frac{5}{2}}}{3};\\
b &= f(w_0) - a\sqrt{w_0-c}.
\end{align*}
For $\delta\le w_0$, we have the $\delta$-smoothing $g(w) = g_1 w +  \frac{1}{2}g_2 w^2 +  \frac{1}{6}g_3 w^3$, where
\begin{align*}
g_1 &= 3a_5\delta^4+a_4\delta^3+a_1;\\
g_2 &= -16a_5\delta^3-6a_4\delta^2+2a_2;\\
g_3 &= 36a_5\delta^2+18a_4\delta+6a_3.
\end{align*}
(The requirement that $\frac{f''(w_0)}{f'''(w_0)} \leq 0$ ensures that $\sqrt{w-c}$ is real-valued for $w > w_0$.)
\hfill  $\Diamond$
\end{example}

\bigskip
And now we are ready to build a function that satisfies the hypotheses of \cref{thm:1_gen}.
\bigskip
\begin{example} \label{ex:thm1_gen}
Following \cref{ex:fg}, let
\begin{displaymath}
f(w) := \left \{
\begin{array}{ll}
a_5w^5+a_4w^4+a_3w^3+a_2w^2+a_1w, & w\le w_0;\\
a\sqrt{w-c}+b,&w > w_0.
\end{array}
\right .
\end{displaymath}
We seek parameters of $f$ for which $g_2\le0$, and all conditions of \cref{thm:1_gen} are satisfied. For $0 \leq w \leq w_0$, we have $f'''(w) = 60a_5 w^2 + 24a_4 w + 6a_3$ and $f''''(w) = 120a_5 w + 24a_4$.  In order to have $f'''(w)$ first decreasing and then increasing on $[0,\delta]$, we require $a_5>0$, $a_4 <0$ and $5a_5\delta+a_4>0$.

For example, choose $\delta=1$, $a_4=-4$, $a_5=1$. It is straightforward to verify that the conditions of \cref{thm:1_gen} now hold. Next, we choose $a_3 = 10$ and $w_0=2$ to have $f'''(w_0)>0$, and we choose $a_2=-50$ to have $f''(w_0)<0$ and $g_2=-16a_5\delta^3-6a_4\delta^2+2a_2<0$. Then we compute the remaining parameters $(a_1,a,b,c)=(132,\frac{4\sqrt{6}}{3}, \frac{332}{3},\frac{11}{6})$.
We can see the difference between $g$ and $f$ in \cref{fig:ex2fg} and the tendency of $f'''$ in \cref{fig:ex2f3}.

\hfill  $\Diamond$
\end{example}
\begin{figure}[H]
  \centering
  \subfloat[$f(w)-g(w)$]{
  \includegraphics[width=0.4\textwidth]{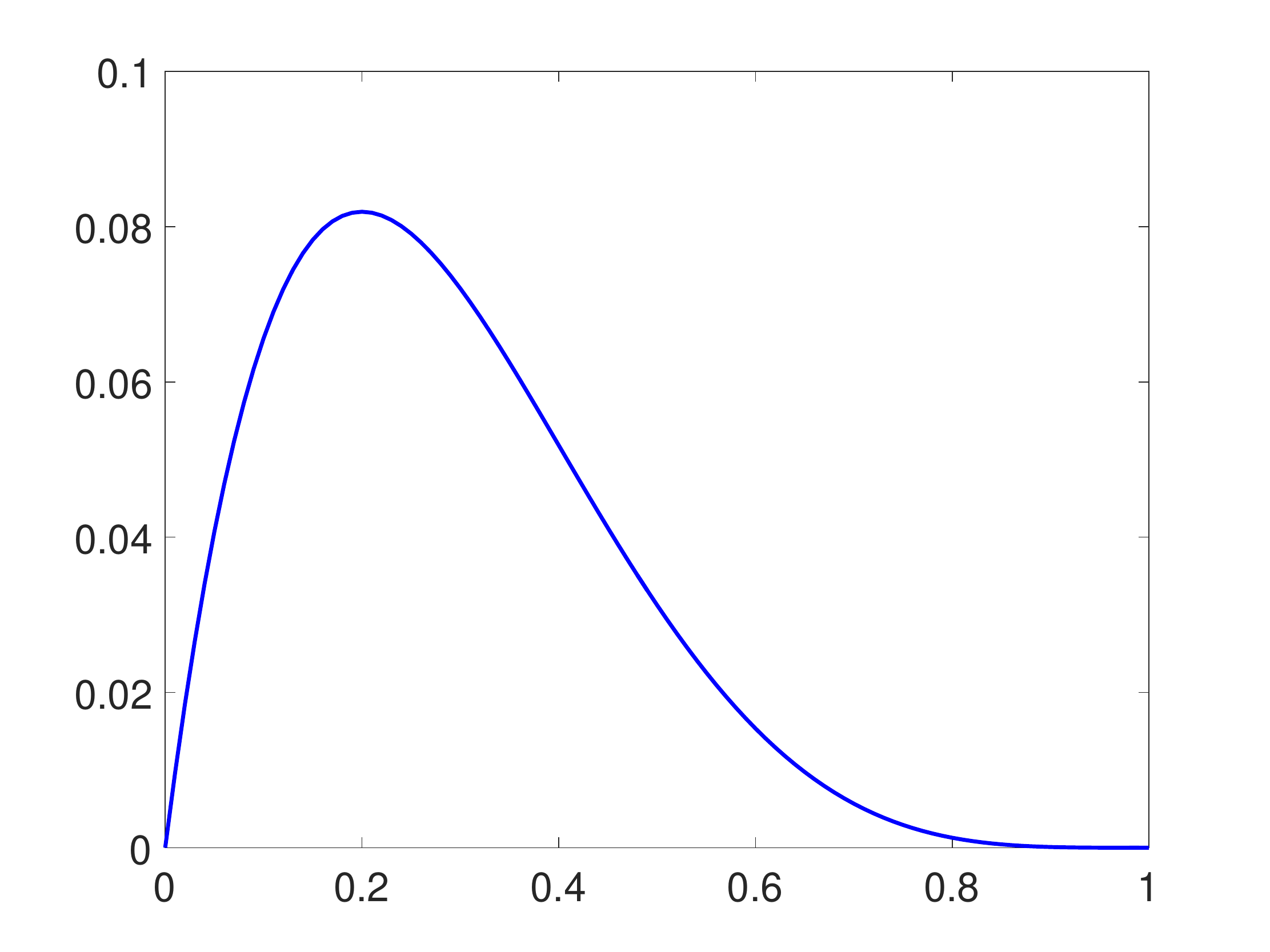}
  \label{fig:ex2fg}
  }
  \subfloat[$f'''(w)-g_3$]{
  \includegraphics[width=0.4\textwidth]{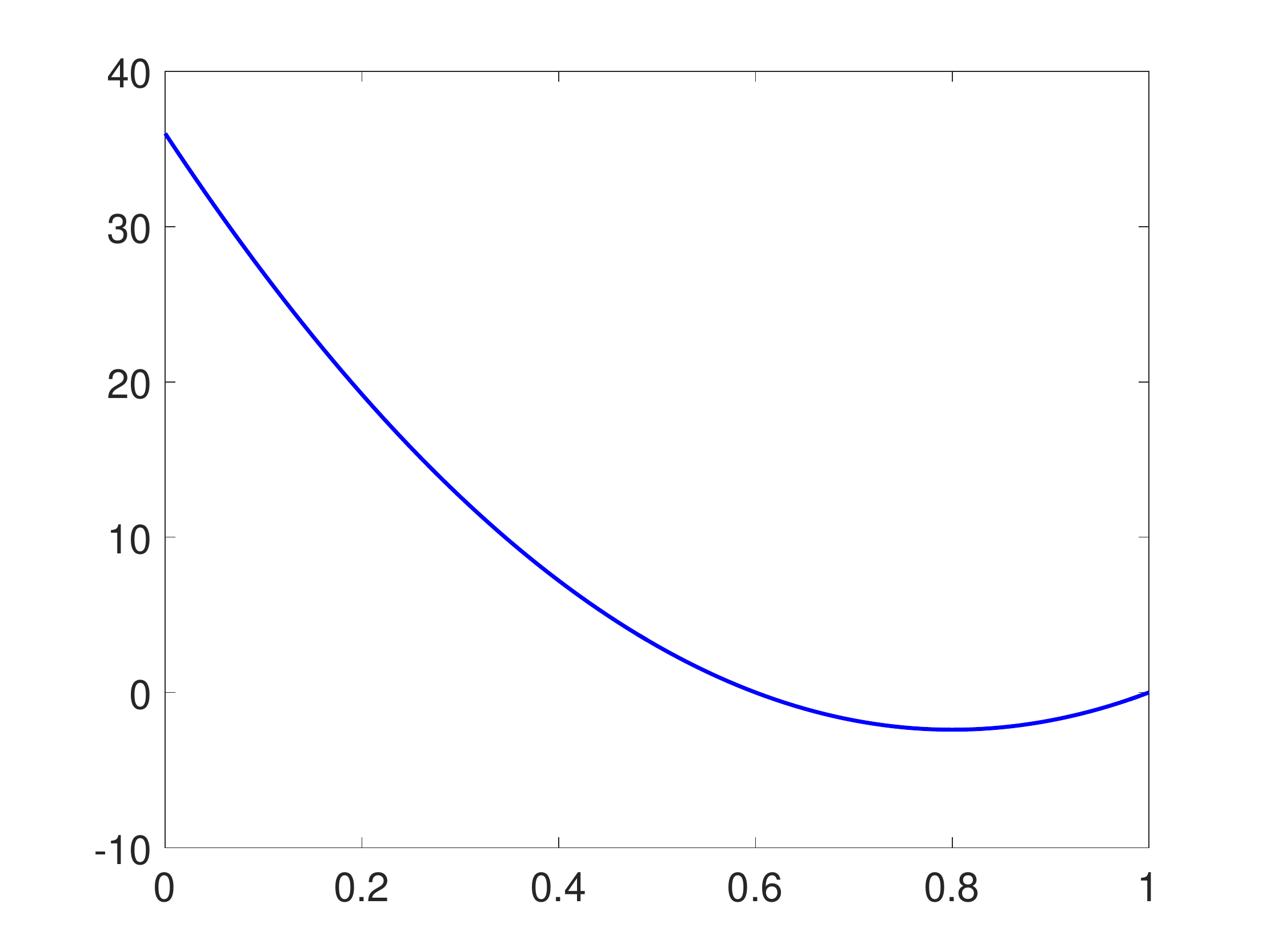}
  \label{fig:ex2f3}
  }
  \caption{$a_5=1$, $a_4=-4$, $a_3=10$, $w_0=2$, $a_2=-50$ in \cref{ex:fg}}
  \label{fig:ex2}
\end{figure}

In the next example we see yet another possibility for the conditions on $f'''$ under which $g$ is increasing and concave. Interestingly, the same function with a different choice of $\delta$ does not have $g \leq f$, but instead provides an example for \cref{prop:ub}, in which we get $g \geq f$.
\bigskip


\begin{example}\label{ex:incdec}
Consider the function $f(w) = -(w+3)e^{-w}+3$, which has the following derivatives:
\begin{align*}
f'(w) &= (w+2)e^{-w};\\
f''(w) &= -(w+1)e^{-w};\\
f'''(w) &= we^{-w};\\
f^{(4)}(w) &= -(w-1)e^{-w}.
\end{align*}
Also $f(0)=0$, and $f$ is increasing, concave and thrice differentiable on $[0,+\infty)$.
Moreover, $f'''(w)$ is increasing on $[0,1]$ and then decreasing on $[1,+\infty)$.

For $\delta=5$,  $g \leq f$ on their common domain (see \cref{fig:delta5}), even though this function satisfies neither the conditions in \cref{thm:1} nor \cref{thm:1_gen}. Instead, $f'''$ is increasing and then decreasing on $[0,\delta]$.

For $\delta=1$, $f'''$ is increasing on $[0,\delta]$. We conclude that $g$ upper-bounds $f$ (see \cref{fig:delta1}) via \cref{prop:ub}.

\begin{figure}[H]
  \centering
  \subfloat[$f(w)-g(w)$]{
  \includegraphics[width=0.4\textwidth]{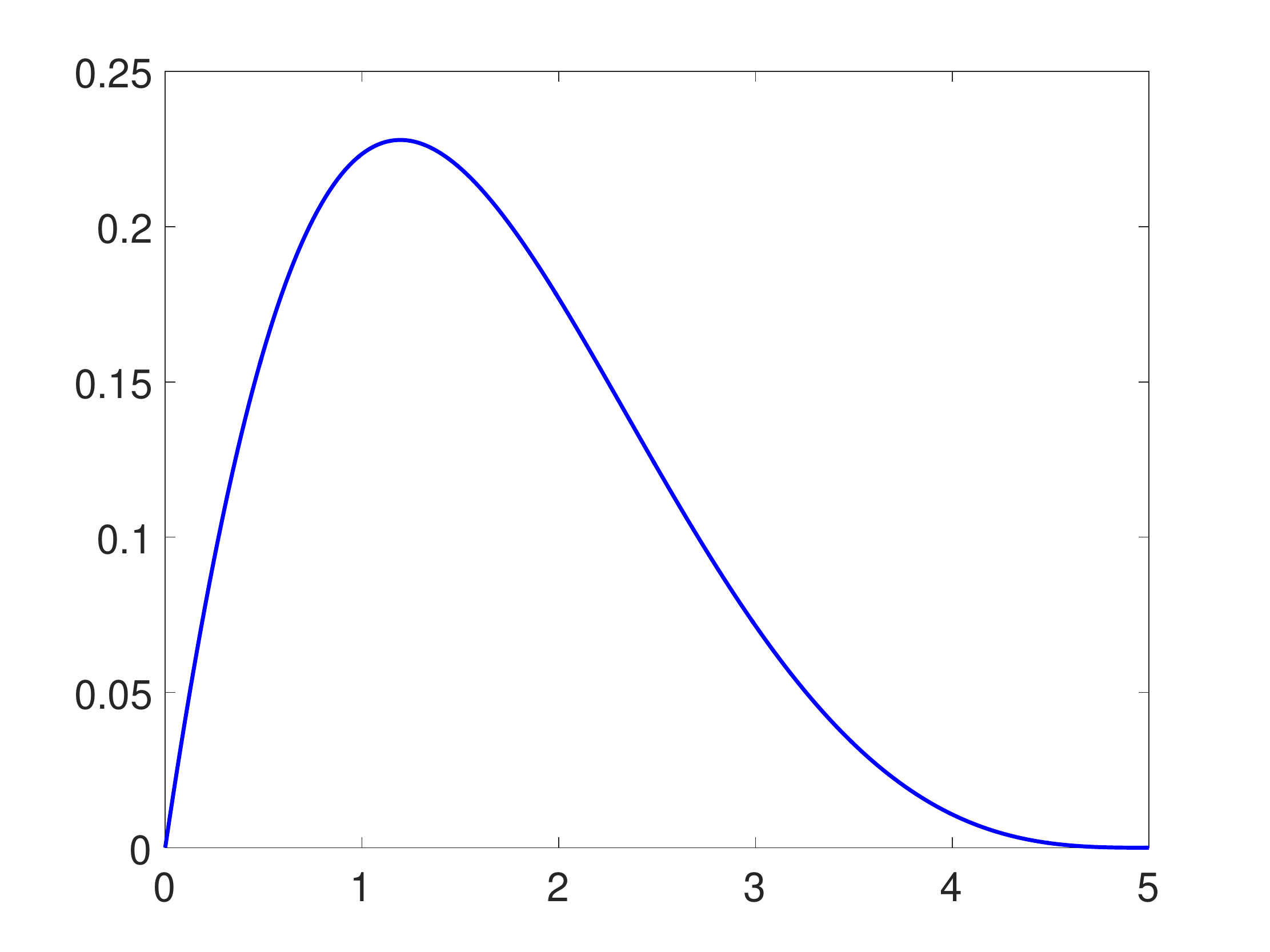}
  \label{fig:delta5}
  }
  \subfloat[$f'''(w)-g_3$]{
  \includegraphics[width=0.4\textwidth]{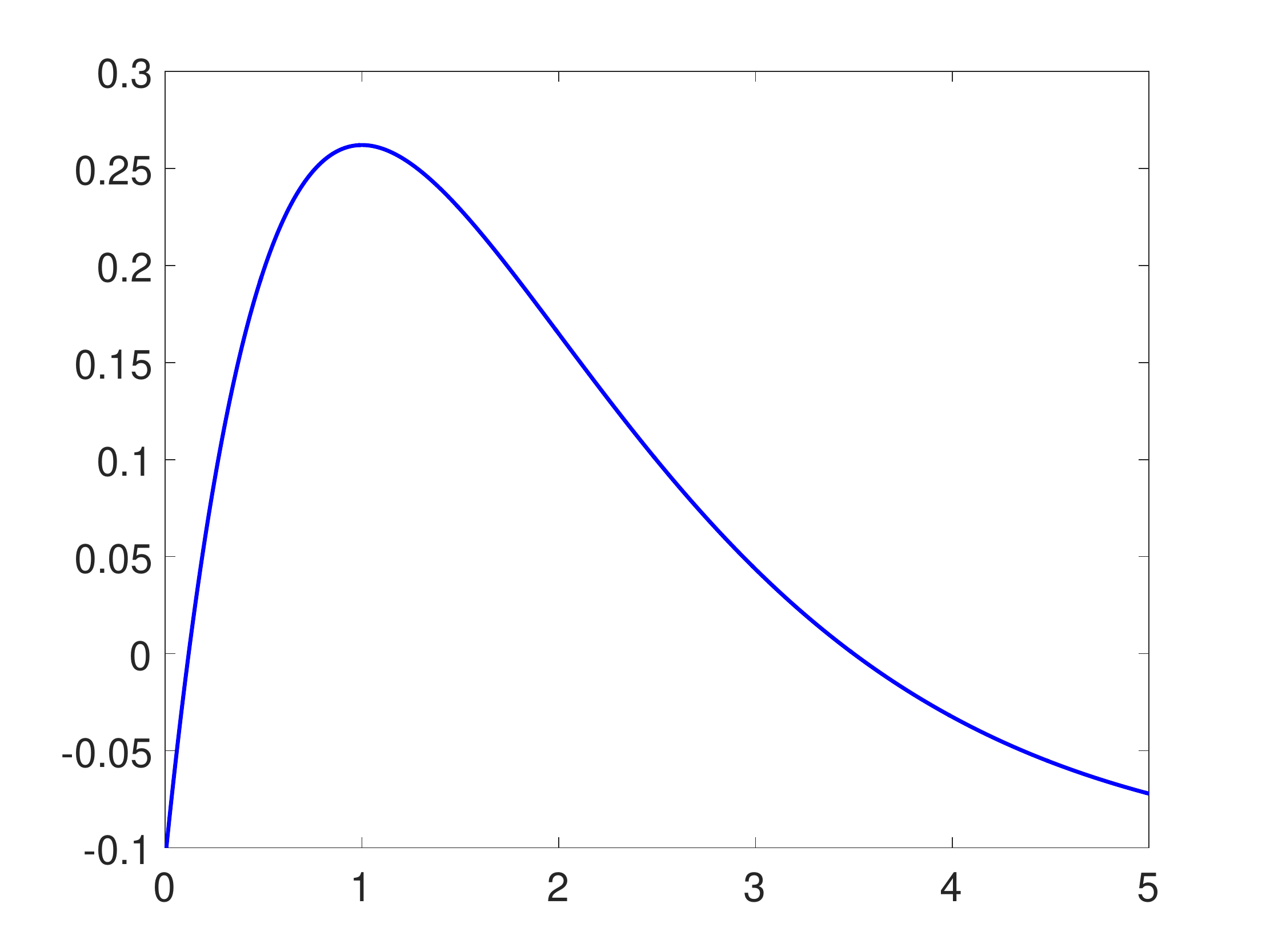}
  \label{fig:delta5f3}
  }
  \caption{$\delta = 5$, $g \leq f$}
  \label{fig:de5}
\end{figure}

\begin{figure}[H]
  \centering
  \subfloat[$f(w)-g(w)$]{
  \includegraphics[width=0.4\textwidth]{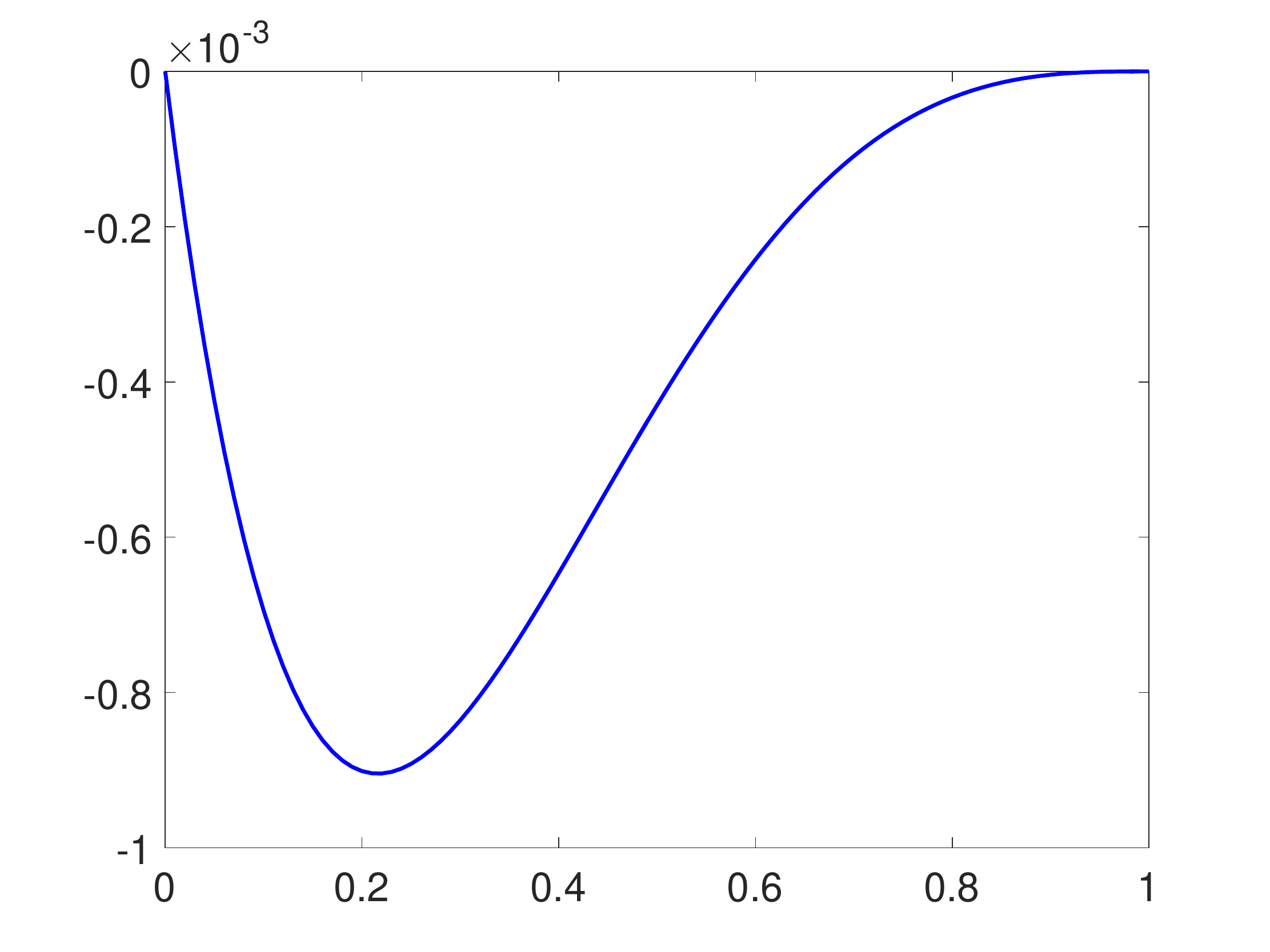}
  \label{fig:delta1}
  }
  \subfloat[$f'''(w)-g'''(0)$]{
  \includegraphics[width=0.4\textwidth]{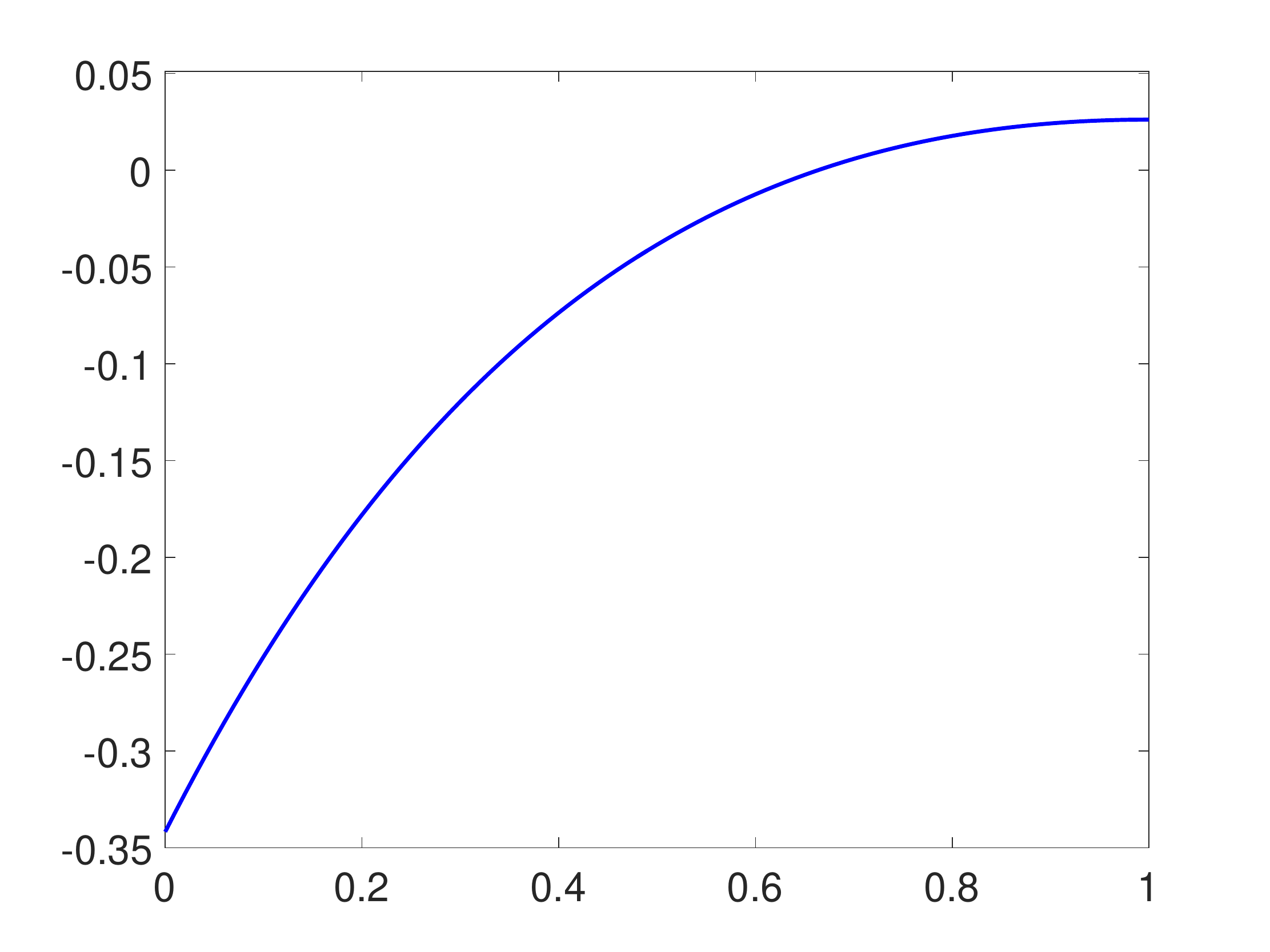}
  \label{fig:delta1f3}
  }
  \caption{$\delta = 1: g \ge f$}
  \label{fig:de1}
\end{figure}
If we add a small positive multiple of the square root function $\sqrt{w}$ to $f$, then we can get other possibilities for the tendency of $f'''$. For example, for $\delta=5$, $\epsilon=5\times 10^{-5}$,  $f(w)+\epsilon\sqrt{w}$ satisfies $g\le f$, while $f'''$ is decreasing, then increasing, then decreasing again.
\phantom{a}\hfill  $\Diamond$
\end{example}

\subsection{Role of the increasing and concave properties}
 \label{roleofTdelta}
\cref{thm:1} suggests that there could be $f$ that are not increasing and concave for which the $\delta$-smoothing of $f$ is a lower bound for $f$. The following simple example realizes such a scenario.

\begin{example}\label{ex:otherf}
Let $\delta=1$, $f(w)=-w^4+6w^2-8w$ is decreasing and convex, and satisfies $f^{(4)}(w)=-1<0$, which implies that $f'''$ is decreasing on $(0,\delta]$, and \cref{thm:1} holds; $g$ is a lower bound for $f$.
\end{example}

Returning to root functions and their relatives,  it would be nice if we could count on the lower bounding $g$ to be increasing and concave whenever \cref{thm:1} applies to an increasing and concave  $f$. In \S\ref{sec:tendency}, we gave a necessary and sufficient condition
\eqref{eq:TSdelta} ($g_2 \leq 0$) for $g$ to be increasing and concave.
So we have the natural question: do we automatically satisfy \eqref{eq:TSdelta} when \cref{thm:1} applies to functions that are increasing and concave?
Unfortunately, the answer to this question is ``no'', as demonstrated by \cref{ex:gg}.

To motivate the development of \cref{ex:gg}, we note that when $f'''(\delta)\ge0$, property \ref{p4} of \cref{prop:thm1}
implies that $g_3 \geq 0$, as well.
So to get an example of a function that satisfies the hypotheses of \cref{thm:1} but $g$ is not increasing and concave, we need to have $f'''(\delta) < g_3 < 0$.
We impose the required properties in the context of the general form presented in \cref{ex:fg}.


\begin{example}\label{ex:gg}
Consider the function $f$ described in \cref{ex:fg}.  We seek parameters of $f$ for which $g_3 < 0$, and all conditions of \cref{thm:1} are satisfied.
  For $0 \leq w \leq w_0$, we have $f'''(w) = 60a_5 w^2 + 24a_4 w + 6a_3$ and $f''''(w) = 120a_5 w + 24a_4$.  In order to have $f'''(w)$ decreasing on $[0,\delta]$ and $f'''(\delta)<0<f'''(w_0)$ for $\delta < w_0$, we require $a_5>0$, $a_4 <0$ and $5a_5\delta+a_4\le0$.

For example, choose $\delta=1$, $a_3=0$, $a_4=-5$, $a_5=1$. It is straightforward to verify that the conditions of \cref{thm:1} now hold. Next, we choose $w_0=3$ to have $f'''(w_0)>0$, and we choose $a_2=-3$ to have $f''(w_0)<0$. By calculating $a_1,~a,~b,~c$, we get an example with $g_3=-54<0$.

Note that $g_2=8>0$, so the associated function $g$ is not increasing and concave.

\cref{fig:ex1} shows the difference between $f$ and $g$, and also the tendency of $f'''$. \cref{fig:fgconvex} shows the derivative and second derivative of $f(w)$ and $g(w)$, respectively, which demonstrates that $f(w)$ is increasing and concave, while $g(w)$ is not.  \hfill  $\Diamond$
\end{example}

We encapsulate the result implied by \cref{ex:gg} as follows:

\begin{observation}
\label{pop:gg}
For an increasing concave function $f$, the hypotheses of \cref{thm:1} do not imply that the smoothing $g$ is increasing and concave, i.e., \eqref{eq:TSdelta} is not implied by the hypotheses of \cref{thm:1},
even for increasing concave $f$.
\end{observation}

\begin{figure}[H]
  \centering
  \subfloat[$f(w)-g(w)$]{
  \includegraphics[width=0.4\textwidth]{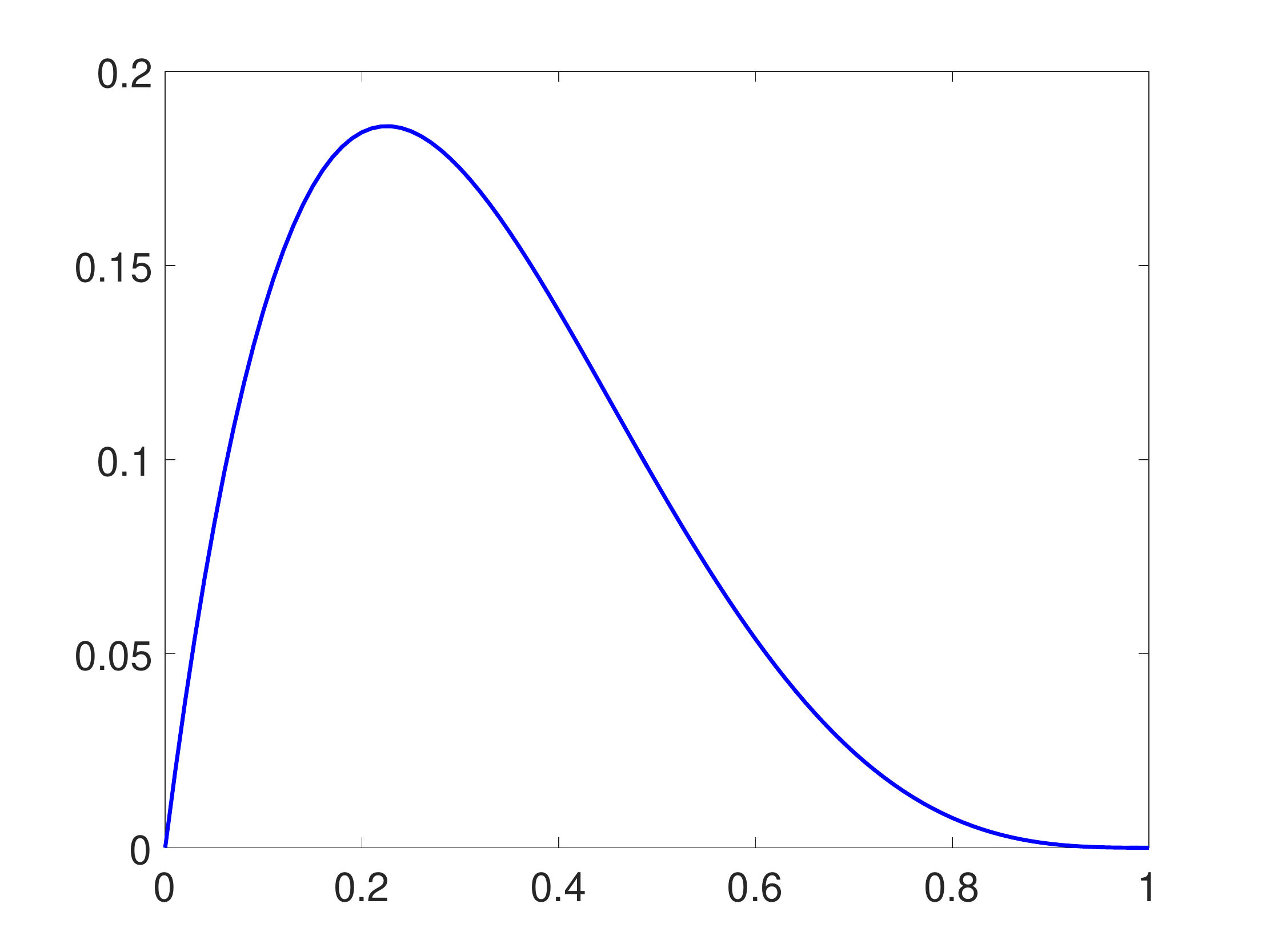}
  \label{fig:ex1fg}
  }
  \subfloat[$f'''(w)-g_3$]{
  \includegraphics[width=0.4\textwidth]{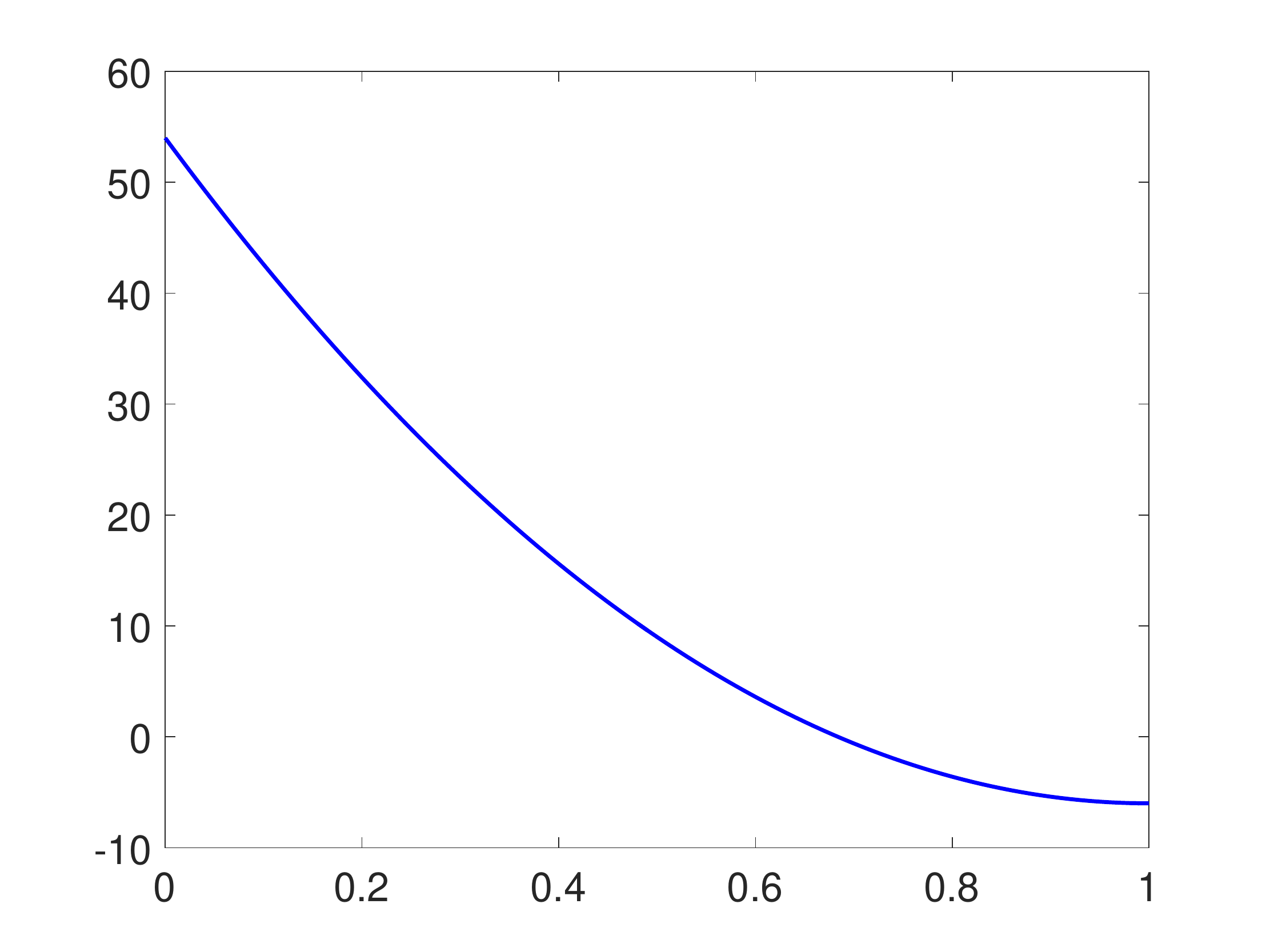}
  \label{fig:ex1f3}
  }
  \caption{$a_5=1$, $a_4=-5$, $a_3=0$, $w_0=3$, $a_2=-3$ in \cref{ex:fg}}
  \label{fig:ex1}
\end{figure}

\begin{figure}[H]
  \centering
  \subfloat[$f'(w),g'(w)$]{
  \includegraphics[width=0.4\textwidth]{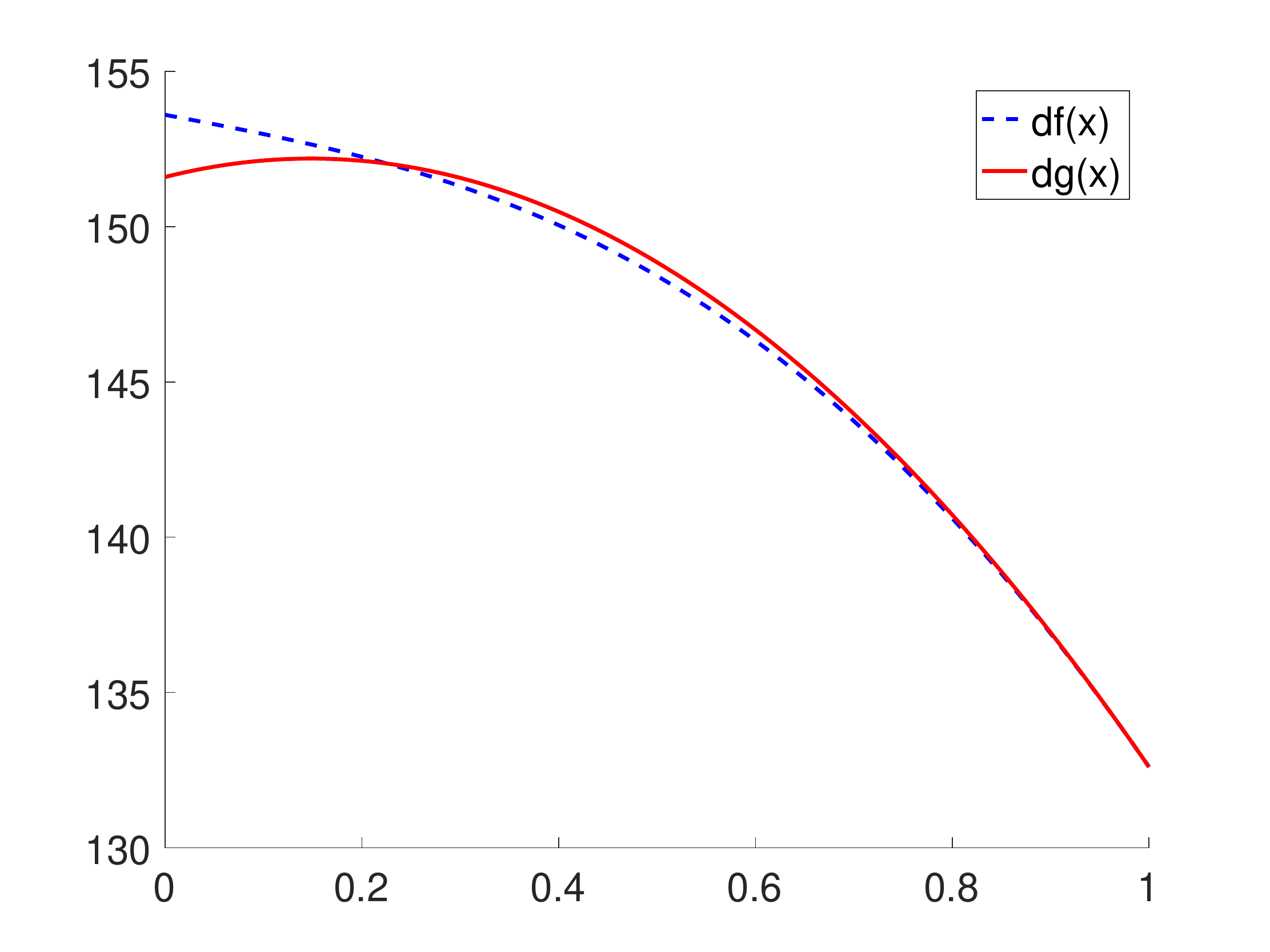}
  \label{fig:dfgconvex}
  }
  \subfloat[$f''(w), g''(w)$]{
  \includegraphics[width=0.4\textwidth]{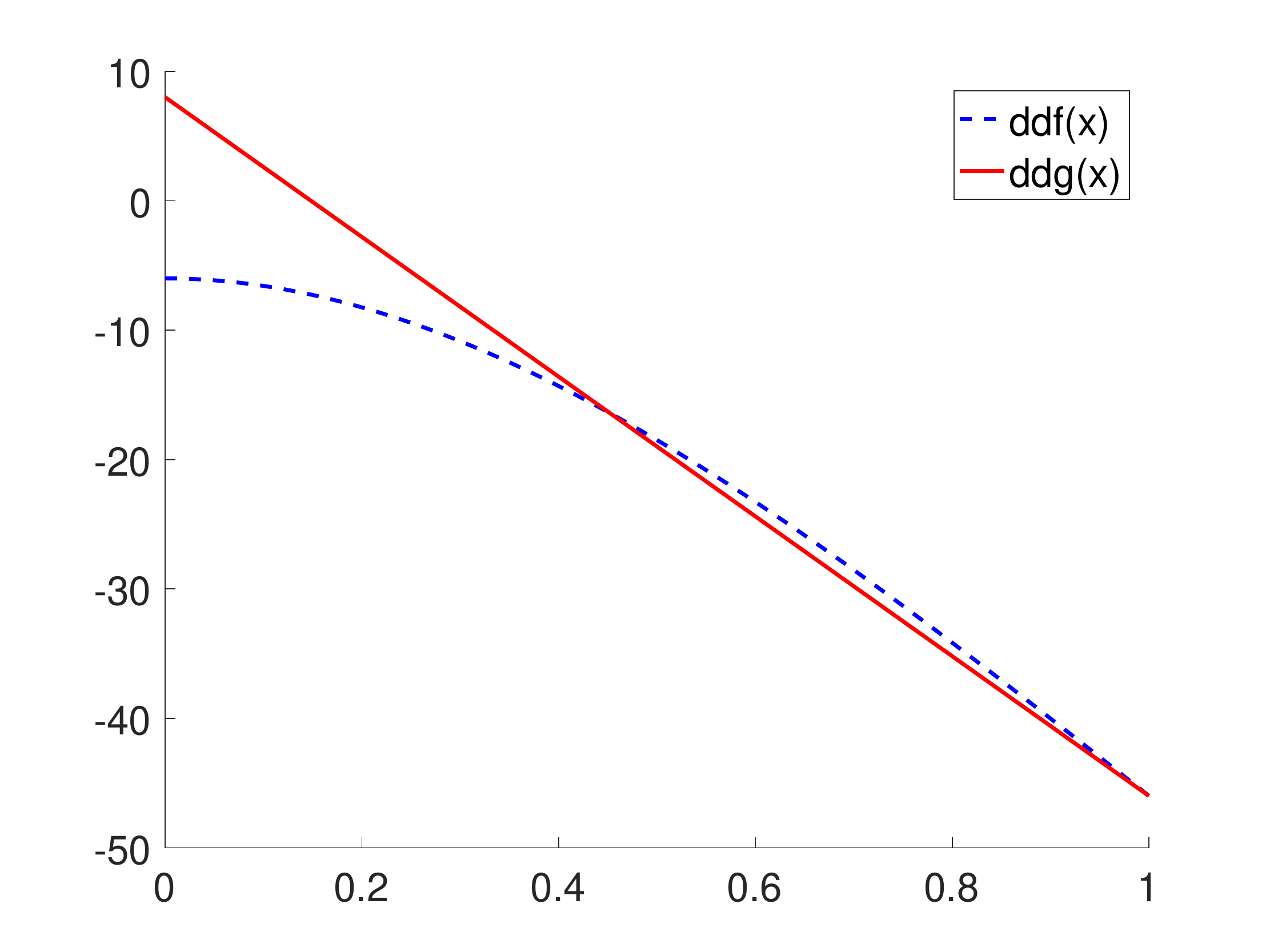}
  \label{fig:ddfgconvex}
  }
  \caption{$a_5=1$, $a_4=-5$, $a_3=0$, $w_0=3$, $a_2=-3$ in \cref{ex:fg}}
  \label{fig:fgconvex}
\end{figure}

\subsection{Monotonicity of \texorpdfstring{$\norm{f-g}_{\infty}$ in $\delta$}{||f-g|| in delta}}\label{fgmonotone}

In \S\ref{monotone}, we demonstrated that the derivative of $g$ at
zero is \emph{decreasing} in $\delta$,
when $f'''$ is decreasing. This is useful for calculating the least value of $\delta$ to obtain a target value for $g'(0)$. In this subsection,
we demonstrate that the worst-case  error of $g$ as an
approximation of $f$ is \emph{increasing} in $\delta$, again when $f'''$ is decreasing. This is useful for calculating the greatest value of $\delta$ to obtain a target value for the worst-case  error of $g$ as an
approximation of $f$. Of course it can be that tolerances for
 the derivative of $g$ at zero and for the worst-case  error of $g$ as an
approximation of $f$ can  be incompatible (i.e., no valid choice of $\delta$ satisfying both). Before continuing,
we note that (i) \cite[Section 5]{lee2016virtuous} obtained results on the average performance of $g$, when $f$ is a root function, and (ii) \cite[Theorem 1, part 6]{DFLV2015},\cite{DFLV2014} obtained  results on the worst-case performance of $g$, when $f$ is the square-root function.

Formally now, we define $F(w):=f(w)-g(w)$, and
$$\norm{f-g}_{\infty} := \max_{w\in [0,\delta]}\abs{f(w)-g(w)} = \max_{w\in[0,\delta]}\abs{F(w)}.$$
Note that $g$ and its coefficients $g_1,g_2,g_3$  are functions of $\delta$, and so $\norm{f-g}_{\infty}$ is also a function of $\delta\in (0,U)$.

\begin{theorem}\label{thm:fgmonotone}
Let $f$ be a univariate function having a domain  $I:=[0,U)$, where $U \in \{ w\in\mathbb{R} ~:~ w > 0\} \cup \{+\infty\}$.  Assume that $f$ satisfies the minimal $\delta$-smoothing requirements for all $\delta>0$ in the domain of $f$.
Suppose further that

\begin{itemize}
  \item $f$ is continuous on $[0,U)$ and thrice differentiable on $(0,U)$;
  \item $f'''$ is decreasing on $(0,U)$.
\end{itemize}
Then $\norm{f-g}_{\infty}$ is increasing on $(0,U)$.
\end{theorem}

\begin{proof}
If $f'''$ is decreasing, then by \cref{thm:1}, $F(w):=f(w)-g(w)>0$ on $(0,\delta)$. Define
\[
w_2 :=  \mbox{argmax} \{F(w) ~:~  w\in[0,\delta]\}.
\]
Then $\norm{f-g}_{\infty} =\max_{w\in [0,\delta]}F(w) = F(w_2)$.
As mentioned in \cref{rk:w}, we can use the same proof technique from \cref{thm:1_gen} to prove now that $F'(w_2)=0$ and $F''(w_2)<0$.\par
Clearly $w_2$ is a function of $\delta$, and we are going to demonstrate that $w_2$ is actually a differentiable function with respect to $\delta$. For any $\delta\in(0,U)$, let $G(x,y):= f'(y)-\frac{g_3(x)}{2}y^2 -g_2(x)y -g_1(x)$. We have $G(\delta,w_2)=F'(w_2)=0$, and
\[
\frac{\partial G(\delta,w_2)}{\partial x} = f''(w_2) - g''(w_2) =F''(w_2)<0.
\]
By the implicit function theorem, there exists a unique differentiable function $y=w_2(x)$ such that $w_2(\delta)=w_2$ and $G(x,w_2(x))=0$ for $x\in N(\delta)$, where $N(\delta)$ is an open interval containing $\delta$.
\par
Therefore,
\begin{align*}
\frac{d F(w_2(\delta))}{d\delta}&=f'(w_2)\frac{d w_2(\delta)}{d\delta} - g'(w_2)\frac{d w_2(\delta)}{d\delta} -\frac{d g_3(\delta)}{d\delta}\frac{w_2(\delta)^3}{6} \\
&\phantom{==}-\frac{d g_2(\delta)}{d\delta}\frac{w_2(\delta)^2}{2} - \frac{d g_1(\delta)}{d\delta}w_2(\delta)\\
&= -\frac{d g_3(\delta)}{d\delta}\frac{w_2(\delta)^3}{6} -\frac{d g_2(\delta)}{d\delta}\frac{w_2(\delta)^2}{2} - \frac{d g_1(\delta)}{d\delta}w_2(\delta)\\
&=-\frac{w_2(\delta)(w_2(\delta)-\delta)^2}{2\delta}(f'''(\delta)-g_3(\delta))>0.
\end{align*}
The second equality follows from $F'(w_2)=0$, the third equality follows from the facts
\begin{equation*}
\frac{d g_1(\delta)}{d\delta} = \frac{\delta}{2}(f'''(\delta)-g_3(\delta)),~ \frac{d g_2(\delta)}{d\delta} =-2(f'''(\delta)-g_3(\delta)),~ \frac{d g_3(\delta)}{d\delta} =\frac{3}{\delta}(f'''(\delta)-g_3(\delta)),
\end{equation*}
and the last inequality follows from $w_2(\delta)>0$ and $f'''(\delta)-g_3(\delta)<0$ (by \cref{prop:thm1}). Thus $\norm{f-g}_{\infty} = F(w_2(\delta))$ is increasing on $(0,U)$.
\end{proof}

\section{Comparison with shift smoothing}
\label{sec:beb}
We wish to compare our smoothing $g$ with the natural and frequently-used \emph{shift smoothing} (for root functions and their relatives): $h(w):=f(w+\lambda)-f(\lambda)$ for $w \in [0,+\infty)$, with $\lambda > 0$ chosen so that $h'(0)$ is numerically tolerable. When the function $f$ that we are considering is globally concave (and
because we assume that $f(0) =0$), $f$ is subadditive, and so $h$ is a lower bound for $f$ on its domain.

Clearly we have $g(0) = h(0) = 0$, and $h(w) \leq f(w) = g(w)$ for $w \geq \delta$, so we are interested in comparing $g$ and $h$ on the interval $(0,\delta)$.   Because $g$ is defined based on a choice of $\delta$, and $h$ is defined based on a choice of $\lambda$, a fair comparison is achieved by making these choices so that their derivatives at $0$ are the same. In this way, both smoothings of $f$ have the same maximum derivative
--- under the hypotheses of our result (\cref{thm:fg_gen}); that is, both smoothings have their derivatives maximized at zero where $f'$ is assumed to blow up, under the hypotheses of \cref{thm:fg_gen}, which imply the hypotheses of \cref{thm:monotone}.


In order to match derivatives at $0$, let $h'(0)=f'(\lambda)=g'(0)=g_1=3f(\delta)/\delta-2f'(\delta)+\delta f''(\delta)/2$.  Then we have
\begin{displaymath}
\hat{\lambda}:=(f')^{-1}\left(3f(\delta)/\delta-2f'(\delta)+\delta f''(\delta)/2\right),
\end{displaymath}
the value of $\lambda$, defined in terms of $\delta$, for which $h'(0)=g_1$.

In \cite{lee2016virtuous}, it is proved
that $h\le g$ for root functions $f=w^p$, with $p=1/q$ for integers $2\le q\le 10,000$. We generalize this result to a class of functions that shares many properties with root functions, and includes all root functions $f(w):=w^p$, for $0<p<1$.
Note that the conditions of \cref{thm:fg_gen} are more restrictive than those of \cref{thm:1}; here we require that $f'''$ is decreasing on $(0,2\delta)$, rather than $(0,\delta]$, and we require that $f'''(w) \geq 0$,  for $w \in (0, 2\delta)$.  This last condition implies that unlike \cref{thm:1} (see \cref{pop:gg}), \eqref{eq:TSdelta} is implied by the hypotheses of \cref{thm:fg_gen}.

\begin{theorem}
\label{thm:fg_gen}
 Let $f$ be a univariate function having a domain  $I:=[0,U)$, where $U \in \{ w\in\mathbb{R} ~:~ w > 0\} \cup \{+\infty\}$. Suppose  that $U\geq \delta/2>0$. Assume that $f$ satisfies the minimal $\delta$-smoothing requirements.
Assume further that
\begin{itemize}
	\item $f$ is continuous, increasing, and strictly concave on its domain;
	\item $f$ is thrice differentiable on $(0,U)$.
\end{itemize}
Moreover, suppose that
\smallskip
\begin{enumerate}[label=(\Roman*),leftmargin=5\parindent,itemsep=1ex]
\labitem{{\it (I)}}{cond1*} $f'''$ is decreasing on $(0,2\delta)$;
\labitem{{\it (II)}}{cond2}  $f'''(w) \geq 0$, \text{ for } $w \in (0, 2\delta)$.
\end{enumerate}
Then
\[
h(w):=f(w+\hat{\lambda})-f(\hat{\lambda})\le g(w),
\]
for $w$ in the domain of $f$,
where the shift constant $\hatl$ is chosen so that $h'(0)=g_1$; i.e.,
$\hatl=(f')^{-1}(g_1)$.
\end{theorem}

\begin{proof}
With condition \ref{cond1*}, $f$ satisfies the hypotheses of \cref{prop:thm1}, so we have all the properties of \cref{prop:thm1}. First, we consider the existence and uniqueness of $\hatl$.  Condition \ref{cond2} and property \ref{p4} imply that $g_3>f'''(\delta)\ge0$, and so $g_1-f'(\delta)=\frac{1}{2}g_3\delta^2-\delta f''(\delta) > 0$.  Therefore, $\lim_{w\to0^+} f'(w) > g_1 > f'(\delta)$, and because  $f'(w)$ is decreasing, there exists exactly one $\hatl$ in $(0,\delta)$ such that $f'(\hatl) = g_1$.

Now consider the function $H:=g-h$, which has
\begin{align*}
H(w) &= g_1 w + \frac{1}{2}g_2 w^2 + \frac{1}{6}g_3 w^3 -f(w+\hat{\lambda})+f(\hat{\lambda});\\
H'(w) &= g_1 + g_2 w + \frac{1}{2}g_3 w^2-f'(w+\hat{\lambda});\\
H''(w) &= g_2 + g_3w - f''(w+\hat{\lambda});\\
H'''(w) &= g_3 - f'''(w+\hat{\lambda}),
\end{align*}
where the coefficients of the associated function $g$ are as usual (repeated here for convenience):
\begin{align*}
g_1&=\frac{3f(\delta)}{\delta}-2f'(\delta)+\frac{\delta f''(\delta)}{2};\\
g_2&=-\frac{6f(\delta)}{\delta^2}+\frac{6f'(\delta)}{\delta}-2f''(\delta);\\
g_3&=\frac{6f(\delta)}{\delta^3}-\frac{6f'(\delta)}{\delta^2}+\frac{3f''(\delta)}{\delta}.
\end{align*}

It is now straightforward to verify that $H(0)=H'(0)=0$, $H(\delta)=f(\delta)-h(\delta)\ge0$, and $H'(\delta)=f'(\delta)-f'(\delta+\hatl)>0$.

Noting that $0 < \hatl < \delta$, by condition \ref{cond2},
\[
H''(\delta) =f''(\delta)-f''(\delta+\hatl)<0,
\]
by condition \ref{cond1*},
\[
H''' \text{ is increasing on } (0,\delta],
\]
and by condition \ref{cond1*} and property \ref{p4} together,
\[
H'''(\delta) =g_3-f'''(\delta+\hatl)>f'''(\delta)-f'''(\delta+\hatl)>0.
\]
Finally, we assert that $H'''(0)<0$ and $H''(0)>0$, which we prove below.


Because $H'''$ is increasing on $[0,\delta]$ with $H'''(0)<0$ and $H'''(\delta)>0$, we see that $H''(w)$ is first decreasing and then increasing on $[0,\delta]$. Because $H''(0)>0$ and $H''(\delta)<0$, there exists exactly one zero of $H''$ on $(0,\delta)$, which we label $v_1$. Thus $H'(w)$ is increasing on $[0,v_1]$ and decreasing on $[v_1,\delta]$. Because $H'(0)=0$ and $H'(\delta)>0$, we see that $H(w)$ is increasing on $[0,\delta]$, and so for $w \in [0,\delta]$, $H(w)\ge H(0)=0$; i.e., $h(w) \leq g(w)$, for $w \in I$.


Now we turn our attention to proving that $H'''(0) < 0$ and $H''(0)>0$.  As the conditions of this theorem are a restriction of those of \cref{thm:1}, we can find the roots of the derivatives of the function $F := f-g$, $0 < w_2 < w_1 < w_0 < \delta$, where $w_0$ is the root of $F'''$, $w_1$ is the root of $F''$, and $w_2$ is the root of $F'$ as in \cref{rk:w}.

From \cref{rk:w}, $F'''$ is decreasing on $(0,\delta)$.  Therefore, to prove that $H'''(0)=g_3-f'''(\hatl)=g'''(\hatl)-f'''(\hatl) < 0$, it suffices to show that $\hatl < w_0$.  Function $f$ satisfies condition \eqref{eq:Tdelta} of \cref{thm:ic}, so $g$ is concave on $(0,\delta]$, and $f'(\hatl)-g'(\hatl)=g'(0)-g'(\hatl)>0$.  Because $F'$ is positive only to the left of $w_2$, we have $\hatl < w_2~ (<w_0)$.

To prove that $H''(0)=g_2-f''(\hatl) > 0$, we demonstrate that $g_2 > f''(\hatl)$, which we accomplish via an inequality that arises as lower and upper bounds on $g'(w_2)-g'(0)$. For the lower bound, because $F'''(w) = f'''(w)-g_3 > 0$ on $[0,w_2]\subset[0,w_0)$, we have
\[
f''(w)> f''(\hatl)+g_3(w-\hatl), \text{ for } w\in[\hatl,w_2].
\]
Therefore, the slope of the secant to $f''$ between the points at $w = \hatl$ and $w = w_2$ is at least $g_3$; i.e.,
\[
g'(w_2)-g'(0)=f'(w_2)-f'(\hatl)> \frac{1}{2}g_3(w_2-\hatl)^2 + f''(\hatl)(w_2-\hatl).
\]

For the upper bound on $g'(w_2) - g'(0)$, we require two observations.  First, by condition \ref{cond1*} and property \ref{p4}, we have $g_3 > f'''(\delta) \geq 0$.  Second, applying
$
g_2+g_3\delta = f''(\delta)\leq 0,
$
we have
$
w_2 < \delta \leq-g_2/g_3.
$
Now we can obtain the upper bound
\begin{align*}
g'(w_2)-g'(0)
&=\frac{1}{2}g_3w_2^2+g_2w_2 \\
&\le \frac{1}{2}g_3(w_2-\hatl)^2+g_2(w_2-\hatl),
\end{align*}
because this inequality is equivalent to
\[
0 \leq -g_3w_2  -g_2 + g_3\hatl/2,
\]
which we verify by applying $g_3>0$ and $w_2 \leq - g_2/g_3$.

Combining these bounds, we have
\[
\frac{1}{2}g_3(w_2-\hatl)^2 + f''(\hatl)(w_2-\hatl)~ <~ g'(w_2)-g'(0) ~\leq~ \frac{1}{2}g_3(w_2-\hatl)^2+g_2(w_2-\hatl),
\]
which reduces to the desired $g_2 > f''(\hatl)$.
\end{proof}

The following corollary demonstrates that \cref{thm:fg_gen} generalizes the result in \cite{lee2016virtuous}, which states that our smoothing $g$ `fairly dominates' the shift smoothing $h$ for root functions of the form $f(w) = w^{1/q}$, with integer $2 \leq  q \leq 10,000$.

\begin{corollary}
Let $f(w): = w^p$, for some $0<p<1$.
Then $h(w)\le g(w)$ for all $w\in [0,+\infty)$.
\end{corollary}

\begin{proof}
According to the derivatives of $f$ in \cref{cor:wp}, it is easy to see that $f(w)$ satisfies the conditions \ref{cond1*} and \ref{cond2} of \cref{thm:fg_gen}. Therefore the conclusion follows.
\end{proof}

Finally, we note that the \cref{thm:fg_gen} also applies to the non-root function that we have explored throughout.

\begin{example}
Let $f(w) = {\rm ArcSinh}(\sqrt{w})=\log\left(\sqrt{w}+\sqrt{1+w}\right)$, for $w \geq 0$. Then $h(w)\le g(w)$ for all $w \geq 0$.
\end{example}

\section{Conclusions}
\label{sec:conclusions}
It may seem like a challenge to automatically identify and apply the techniques
that we have presented. But in the context of global optimization aimed at factorable
formulations, the algorithm/software designer has a limited number of library functions
to analyze. Furthermore, even in a fully extensible system, we could automatically apply major parts of
our ideas. For example, once a univariate function $f$ has been identified to satisfy $f(0)=0$, $f$ is increasing and concave on say $[0,+\infty)$, $f$ is twice differentiable on
all of $(0,+\infty)$, but $f'(0)$ undefined or intolerably large, then the rest of our methodology (i.e., calculating $g$ and identifying its properties) can be done
automatically. A start has been made on making accommodations for our methodology in \verb;SCIP;.
Hopefully we will see more advances in such a direction, contributing to
the overall goal of making MINLO software more robust and useful.

From a mathematical point of view, still aiming at potential impact on
MINLO software, we could look at functions $f$ with domain being a 2-variable
polyhedron $P$, where $f$ is nice and smooth on the interior of $P$,
but not differentiable on part of the boundary of $P$.


\section*{Acknowledgments}
The authors thank an anonymous reviewer who suggested the
problem solved in \S\ref{fgmonotone}.
J. Lee was supported in part by ONR grant N00014-17-1-2296.
Additionally, part of this work was done while J. Lee was visiting the Simons Institute for the Theory of Computing. It was partially supported by the DIMACS/Simons Collaboration on Bridging Continuous and Discrete Optimization through NSF grant \#CCF-1740425.

\bibliographystyle{alpha}
\bibliography{smooth}
\end{document}